\documentclass[10pt]{paper}%
\usepackage[latin1]{inputenc}
\usepackage{amsmath,amssymb}
\usepackage{fullpage}
\usepackage{amsthm}
\usepackage{arydshln}
\usepackage{multirow}
\usepackage{color}
\usepackage{pifont}
\usepackage{amsmath}
\usepackage{graphicx}
  \usepackage[all]{xy}
  \usepackage{url}
\usepackage{amsfonts}
\usepackage{amssymb}
\usepackage{ytableau}%

\providecommand{\U}[1]{\protect\rule{.1in}{.1in}}
\providecommand{\U}[1]{\protect\rule{.1in}{.1in}}
\providecommand{\U}[1]{\protect\rule{.1in}{.1in}}
\providecommand{\U}[1]{\protect\rule{.1in}{.1in}}
\providecommand{\U}[1]{\protect\rule{.1in}{.1in}}
\newcommand{\ulambda}{{\boldsymbol{\lambda}}}

\newcommand{\umu}{{\boldsymbol{\mu}}}
\newcommand{\unu}{{\boldsymbol{\nu}}}

\newcommand{\uemptyset }{{\boldsymbol{\emptyset}}}
\newtheorem{Th}{Theorem}[subsection]

\newtheorem{Prop}[Th]{Proposition}

\theoremstyle{remark}
\newtheorem{Rem}[Th]{Remark}{\rmfamily}
\theoremstyle{definition}
\newtheorem{Def}[Th]{Definition}{\rmfamily}
\newtheorem{exa}[Th]{Example}{\rmfamily}

\newcommand\blfootnote[1]{%
  \begingroup
  \renewcommand\thefootnote{}\footnote{#1}%
  \addtocounter{footnote}{-1}%
  \endgroup
}
\begin{document}

\title{Mullineux involution and crystal isomorphisms}
\author{Nicolas Jacon}
\maketitle
\date{}
\blfootnote{\textup{2010} \textit{Mathematics Subject Classification}: \textup{20C08,05E10}} 
\begin{abstract}
We develop a new approach for  the computation of the Mullineux involution for the symmetric group and its Hecke algebra 
 using the notion of crystal isomorphism and the Iwahori-Matsumoto involution for the affine Hecke algebra of type $A$.  As a consequence, we obtain 
  several new elementary combinatorial algorithms for its computation, 
  one of which  is equivalent to Xu's algorithm (and thus Mullineux' original algorithm). We thus obtain a simple  interpretation of these algorithms and a  new elementary proof 
   that they indeed compute the Mullineux involution. 
 
\end{abstract}

\section{Introduction}

The Mullineux problem is a long standing problem in  the representation theory of the symmetric groups which has been studied by various authors since the end of the $70$'s. Let $\mathfrak{S}_n$ be the symmetric group on $n$ letters with $n>1$. It is known that the irreducible representations of $\mathfrak{S}_n$ over the field of complex numbers  are naturally labeled by the partitions of $n$ (the sequences of non increasing positive integers of total sum $n$.)
$$\operatorname{Irr}_{\mathbb{C}} (\mathfrak{S}_n)=\{ \rho_{\lambda} \ |\ \lambda \text{ partition of }n\}.$$
 The characters and the dimensions of these representations may also been easily computed thanks to the combinatorics of partitions. There are exactly two non isomorphic representations of $\mathfrak{S}_n$  with dimension $1$: the trivial representation which is labeled by the partition $(n)$ and the sign representation $\varepsilon$, labeled by the partition  $(\underbrace{1.\ldots.1)}_{n\text{ times}}$.  As a consequence,  if $\lambda$ is a partition of $n$, there exists another partition $\mu$ such that $\rho_{\mu}\simeq \varepsilon \otimes \rho_{\lambda}$. It is natural to ask how one can compute  $\mu$ from $\lambda$. The result is that $\mu$ is the conjugate partition of $\lambda$  which  is defined  by interchanging rows and columns in the Young
diagram of $\lambda$ (the Young diagram of $\lambda$ is the finite collection of boxes  arranged in left-justified rows, with $\lambda_k$  boxes in the $k$th row for all $k\geq 1$.)
 
Of course, all the above questions and problems arise when we replace $\mathbb{C}$ by  an arbitrary field $k$ and in particular by  a field of characteristic $p>0$.  In this case, the irreducible representations have first been constructed in \cite{J1}. They are   labeled by a subset of partitions called the set of $p$-regular partitions   the partitions  of $n$ where the non zero parts are not repeated $p$ or more times.
$$\operatorname{Irr}_{k} (\mathfrak{S}_n)=\{ \widetilde{\rho}_{\lambda} \ |\ \lambda \text{ $p$-regular partition of }n\}.$$
 We also have two one-dimensional representations: the trivial representation and the sign representation $\varepsilon$ and they are non isomorphic if and only if $p\neq 2$. 
By contrast, we still not even know how to compute the dimensions of these representations in general. The other mentioned problem  still makes sense in this context. Namely, if $\lambda$ is a  $p$-regular partition then there exists a unique  $p$-regular partition $\mu$ such that $ \widetilde{\rho}_{\mu}\simeq \varepsilon \otimes  \widetilde{\rho}_{\lambda}$. If we set $m_p (\lambda):=\mu$, we thus obtain   an involution $m_p$ on the set of $p$-regular partitions.

 If $p=2$ then it is clear that $m_p=\operatorname{Id}$ (because then $\varepsilon$ is nothing but the trivial representation)  but in general, it is difficult to describe $m_p$.  In fact, this map  may even be defined in the context of Hecke algebras of type $A$ at a $p$-root of unity. In this case,  $p$ do not need to be a prime but just a positive integer (greater than $2$).    The associated involution that we obtain coincides with $m_p$ if $p$ is prime. A natural  problem is thus to find an explicit description of this  involution  $m_e$ on  
 the set of $e$-regular partitions for all $e\in \mathbb{N}_{>1}$.  This is the main subject of the present paper. 

 In \cite{Mu}, Mullineux has first given a conjectural algorithm for computing this involution (which will be called the Mullineux involution in the sequel).  
 Later, another equivalent algorithm has been given by Xu \cite{Xu1,Xu2}.  In  
 \cite{K},  
 Kleshchev gave another combinatorial recursive algorithm for computing the Mullineux involution  but it was not clear at that time why this algorithm would be equivalent 
  to the Mullineux (and the Xu's) algorithm.  Ford and Kleshchev   gave a proof of this fact later in \cite{FK}. Another proof was given in \cite{BO} by Bessenrodt and Olsson. In \cite{BK}, Brundan and Kujawa gave another proof using works by Serganova  on the general linear supergroup.  
   We also note that recently, Fayers \cite{F}  has given another way for computing the involution. 
  
  The aim of this paper is to present several  elementary  combinatorial (and recursive) algorithms for the computation of the involution using the Kleshchev result. These algorithms are based on the results of  \cite{JL,JL2} and on the following points:
  \begin{enumerate}
  \item Each simple module for the Hecke algebra of type $A$ labeled by an $e$-regular partition  of rank $n$ can be seen as a simple module for the affine Hecke algebra of type $A$.
  \item The Mullineux map at the level of Hecke algebra coincide with the so called Iwahori-Matsumoto involution for the  affine Hecke algebra of type $A$.
  \item The Iwahori-Matsumoto involution may be computed using an analogue involution at the level of Ariki-Koike algebras associated to a multicharge ${\bf s}\in \mathbb{Z}^l$.
  \item This later involution may be computed using the Mullineux involution for Hecke algebras of type $A$ on $e$-regular partitions with rank (strictly) less than $n$.
  \end{enumerate}
As a consequence, to compute the image of an $e$-regular partition  of rank $n$ under the Mullineux involution, we are reduced to compute 
 several images of  $e$-regular partitions  of rank strictly less than $n$ under the Mullineux involution. This thus gives a recursive algorithm to solve our problem. In fact, depending on the {\it multicharge}, we choose for 
  our Ariki-Koike algebras, we obtain several different algorithms. It turns out that for a particular choice of multicharge, our algorithm  is equivalent to Xu's algorithm. This thus gives a new elementary proof for the fact that the Mullineux and the Xu's algorithm give an answer for the Mullineux problem. This also gives a new  interpretation of these algorithms (another  interpretation is also given in \cite{BK}).
  
  The paper will be organized as follows. In section $2$, we recall some basic facts on  the representation theory of affine Hecke algebras of type $A$ and of Ariki-Koike algebras. We also recall several  results coming from  \cite{JL,JL3} 
   concerning the labelling of the simple modules for these algebras and the relations between them. Section $3$ introduces the Mullineux and the Iwahori-Matsumoto involutions and shows how these two maps are related. In section $4$, we study  combinatorial properties of partitions and multipartitions which will be used in the following sections. Section $5$ 
 gives the algorithms we get for computing the Mullineux involution. The last section shows that Xu's algorithm can be seen as one of our algorithm. \\
 \\

\noindent {\bf Acknowledgement:} The author  thanks C\'edric Lecouvey for fruitful discussions on the subject of this paper. The author is supported by ANR project JCJC ANR-18-CE40-0001.

\section{Hecke algebras}

In this first section, we recall the definitions of  the affine Hecke algebra of type $A$ and of the Ariki-Koike algebras.   We then give a brief overview of their  representation theories.  Finally, we explain the relations between the known parametrizations of the simple modules for these algebras.  The main references for these parts are
 \cite{A} and \cite{GJ}. 

%
%
%
%
%
%
%
%

\subsection{Affine Hecke algebra of type $A$}

Let $n\in \mathbb{Z}_{>0}$. Let  $q \in \mathbb{C}^{*}$ be a primitive root of unity of order $e>1$. The {\it Iwahori-Hecke algebra} $H_n (q)$ of type $A$ is the unital associative $\mathbb{C}$-algebra 
 generated by $T_0$, $T_1$, \ldots, $T_{n-1}$ and subject to the relations:
 $$\begin{array}{rcl}
 T_i T_{i+1} T_i &=& T_{i+1} T_i T_{i+1} \ (i=1,\ldots,n-2),\\
 T_i T_j &=&T_j T_i \ (|i-j|>1),\\
 (T_i-q)(T_i+1)&=&0 \ (i=1,\ldots,n-1).
 \end{array}$$
The  {\it affine Hecke algebra} $\widehat{H}_n (q)$ is the unital associative $\mathbb{C}$-algebra which is isomorphic to 
$$H_n (q)\otimes_{\mathbb{C}} \mathbb{C} [X_1^{\pm 1} ,\ldots, X_n^{\pm 1} ],$$
as a $\mathbb{C}$-vector space and such that $\widehat{H}_n (q)$ and $\mathbb{C} [X_1^{\pm 1} ,\ldots, X_n^{\pm 1} ]$ are both subalgebras of 
 $\widehat{H}_n (q)$  with the following   additional relations:
$$T_i X_i T_i = qX_{i+1} ,\ T_i X_j =X_i T_j,$$
for all $(i,j)\in \{1,\ldots,n-1\}^2$ with $i\neq j$.

We denote by $\operatorname{Mod}_n$ the category of finite dimensional $\widehat{H}_n (q)$-modules  such that 
 for all $j=1,\ldots,n$, the eigenvalues of the $X_j$ are power of $q$.  The simple objects $\operatorname{Irr} (\widehat{H}_n (q))$ in $\operatorname{Mod}_n$ can be naturally labeled by the set of aperiodic multisegments that we now define:

\begin{Def}
Let $l\in \mathbb{N}_{>0}$ and let $i \in \mathbb{Z}/e\mathbb{Z}$. The {\it segment} of length $l$ and head $i$ is the sequence of consecutive {\it residues} (i.e elements of $\mathbb{Z}/e\mathbb{Z}$, identified with $\{0,1,\ldots,e-1\}$)   $[i,i+1,\ldots, i+l-1]$ in $\mathbb{Z}/e\mathbb{Z}$. The 
 residue $i\in \mathbb{Z}/e\mathbb{Z}$ is then called the {\it head} of the segment and the residue $i+l-1$ the {\it tail} of the segment.  
 A {\it multisegment} is a formal sum of segments. A multisegment is said to be {\it aperiodic} if for every $l\in \mathbb{Z}_{>0}$, there exists $i\in \mathbb{Z}/e\mathbb{Z}$ such that there is 
  no segment with length $l$ and tail $i$ appearing in the multisegment. We denote by $\mathfrak{M}_e$ the set of aperiodic multisegments. 
   The length of a multisegment is the sum of the lengths of the the segments appearing in it and is denoted by $|\psi|$. We denote by 
   $\mathfrak{M}_e (n)$ the set of aperiodic multisegments of length $n$. 
\end{Def}

\begin{exa}
For $e=3$, the multisegment:
$$[0,1,2,0]+[0]+[1]+[1,2]+[2,0]$$
is an aperiodic multisegment of length $10$ where as 
$$[0,1,2,0]+[0]+[0,1]+[1,2]+[2,0]$$
is a multisegment of length $10$ which is not aperiodic. 
\end{exa}

By the geometric realization of 
 $\widehat{H}_n (q)$ by Chriss and Ginzburg \cite{CG}, we know  that one may naturally label the simple modules in  $\operatorname{Mod}_n$  
  by the set $\mathfrak{M}^e (n)$  of aperiodic multisegments of length $n$.  We thus have:
$$\operatorname{Irr} (\widehat{H}_n (q))=\{ L_{\psi} \ |\ \psi \in \mathfrak{M}^e (n)\}$$

\subsection{Ariki-Koike algebras}\label{act}
 As above, we fix a primitive root of unity $q\in \mathbb{C}^*$ of order $e>1$. Let $P_l:=\mathbb{Z}^l$ and let $\{ z_i\ |\ i=1,\ldots ,l\}$ be the canonical basis of $P_l$.  Let $\mathfrak{S}_l$ be the symmetric group generated by the transpositions $\sigma_i:=(i,i+1)$ for $i=1,\ldots,l-1$. 
The extended affine symmetric group  $\widehat{\mathfrak{S}}_l$ is  the semidirect product  $P_l \rtimes \mathfrak{S}_l$ with  the relations  given by $\sigma_i z_j=z_j \sigma_i$ for $j\neq i,i+1$ and $\sigma_i z_i \sigma_i=z_{i+1}$ for $i=1,\ldots,l-1$ and $j=1,\ldots,l$.   This group is generated by the $\sigma_i$ for $i=1,\ldots,l-1$ and by 
 $\tau:=z_l \sigma_{l-1}\ldots \sigma_1$ (see \cite[\S 5.1]{JL}.)

It  acts faithfully on $\mathbb{Z}^l$   as follows: for any ${\bf s}=(s_{1},\ldots ,s_{l})\in 
\mathbb{Z}^{l}$: 
$$\begin{array}{rcll}
\sigma _{c}.{{\bf s}}&=&(s_{1},\ldots ,s_{c-1},s_{c+1},s_{c},s_{c+2},\ldots ,s_{l})&\text{for }c=1,\ldots,l-1 \text{ and }\\
z_i.{{\bf s}}&=&(s_{1},s_{2},\ldots,s_i+e,\ldots ,s_{l})&\text{for }i=1,\ldots,l.
\end{array}$$
and we have 
$$\tau.{\bf s}=(s_2,\ldots,s_l,s_1+e).$$
 Let $\mathfrak{s}$ be an orbit with respect to the above action and let  ${\bf s}:=(s_1,\ldots,s_l) \in \mathbb{Z}^l$ be an element in this orbit.  
 The Ariki-Koike algebra $\mathcal{H}_n^{\mathfrak{s}} (q)$ 
 is the quotient $\widehat{H}_n (q)/ I_{\mathfrak{s}}$ where $I_{\mathfrak{s}}:=\langle \prod_{1\leq j\leq l} (X_1-q^{s_j}) \rangle $.
  If $l=1$, this is a Hecke algebra of type $A$ (of finite type), and if $l=2$ a Hecke algebra of type $B$ (of finite type). One can see that the above algebra is well defined and  depends only on 
   the orbit of ${\bf s}$ modulo the action of  $\widehat{\mathfrak{S}}_l$ (and on $q$).

  The representation theory of this algebra has been  intensively studied in a  number of works. We refer to \cite{A,GJ} and the references theirin. We will only recall what is needed for the results of the present paper.   The analogues of the multisegments in the context of Ariki-Koike algebras are the multipartitions that we now define.  For this, let us give some additional combinatorial definitions.  
  
 A {\it partition} is a nonincreasing
sequence $\lambda=(\lambda_{1},\cdots,\lambda_{m})$ of nonnegative
integers. One can assume this sequence is infinite by adding parts equal to
zero. The {\it rank}  of the partition is by  definition the number $|\lambda|=\sum_{1\leq i\leq m} \lambda_i$. 
 We say that $\lambda$ is a partition of $n$, where $n=|\lambda|$. By convention, the unique partition of $0$ is the empty partition $\emptyset$. 

More generally, for $l\in \mathbb{Z}_{>0}$, an {\it $l$-partition} $\ulambda$ of $n$ is a sequence of $l$ partitions $(\lambda^1,\ldots,\lambda^l)$ 
 such that the sum  of the ranks of the $\lambda^j$ is $n$. The number $n$ is then called  the {\it rank} of $\ulambda$ and it is denoted by $|\ulambda|$.  The set of $l$-partitions is denoted by $\Pi^l$ and the set  of $l$-partitions of rank $n$  is denoted by $\Pi^l (n)$.  Let $\ulambda$ be an $l$-partition. The {\it nodes} or the {\it boxes}  of $\ulambda$ are by definition the elements of   the Young diagram of $\ulambda$:
$$[\ulambda]:=\{ (a,b,c)  \ | \ a\geq 1,\ c\in \{1,\ldots,l\},\ 1\leq b\leq \lambda_a^c\} \subset \mathbb{Z}_{>0}\times 
  \mathbb{Z}_{>0} \times \{1,\ldots,l\}.$$
  The {\it content}  of  a node $\gamma=(a,b,c)$ of $\ulambda$ is the element $b-a+s_c$  of $\mathbb{Z}$ and the residue is the content modulo $e\mathbb{Z}$.  
  If $l=1$ (that is when we consider a partition instead of a multipartition), then the Young diagram is identified with a subset of $\mathbb{Z}_{>0}\times 
  \mathbb{Z}_{>0}$ in an obvious way.

  Since the works of Ariki and Lascoux-Leclerc-Thibon, it is known that the representation theory of these algebras  is closely related to the representation theory of quantum groups. In particular, one can naturally label the simple modules by the crystal basis of a certain integrable representation for the quantum group of affine type $A$. 
  We will not give the details of all the consequences of this fact but we summarize this below. Again, we refer to \cite{GJ} for a complete study.  
   For all choice of ${\bf s}\in {\mathfrak{s}}$, we can define a certain subset of $l$-partitions which are called Uglov $l$-partitions and which are denoted by $\Phi_{e,{\bf s}}(n)$.   
    These classes of  multipartitions, which strongly depends on  the choice of ${\bf s}$, can all be seen as non trivial generalizations of the set of  $e$-regular partitions:
   
   \begin{itemize}
   \item For all $s\in \mathbb{Z}$, we define:
   $$\mathcal{A}_e^l[s] :=\{ (s_1,\ldots,s_l) \in \mathbb{Z}^l \ |\  s_1=s\leq s_2  \leq \ldots s_l <s+e\}.$$
  This is a fundamental domain  for the action of  $\widehat{\mathfrak{S}}_l$ on $\mathbb{Z}^l$. 
If  ${\bf s}\in \mathcal{A}_e^l[s]$,     then the $l$-partitions in $\Phi_{e,{\bf s}} (n)$  are known as FLOTW $l$-partitions and they have a non recursive definitions: 
 we have $\ulambda=(\lambda^1,\ldots,\lambda^l)\in \Phi_{{\bf s},e} (n)$ 
 if and only if:
 \begin{enumerate}
 \item For all $j=1,\ldots,l-1$ and $i\in \mathbb{Z}_{>0}$, we have:
 $$\lambda_i^j\geq \lambda_{i+s_{j+1}-s_j}^{j+1}.$$
 \item For all $i\in \mathbb{Z}_{>0}$, we have:
 $$\lambda_i^{l}\geq \lambda_{i+e+s_{1}-s_l}^{1}.$$ 
 \item For all $k\in \mathbb{Z}_{>0}$, the set 
 $$\{ \lambda_i^j-i+s_j+e\mathbb{Z}\ |\ i\in \mathbb{Z}_{>0},\ \lambda_i^j=k, j=1,\ldots,l\},$$
 is a proper subset of $\mathbb{Z}/e\mathbb{Z}$. 
 
 \end{enumerate}
 
   \item If ${\bf s}$ satisfies for all $i=1,\ldots,l-1$, $s_{i+1}-s_i>n-1$  (we say that ${\bf s}$ is {\it very dominant}, it is also sometimes referred as the ``asymptotic case'' in the literature) then the set $\Phi_{e,{\bf s}} (n)$ is known as the set  Kleshchev $l$-partitions.  If ${\bf s}''$ satisfy the same property, then the associated set $\Phi_{e,{\bf s}"} (n)$ is the same.

   \item If $l=1$, the set  $\Phi_{e,(s)} (n)$ is simply the set of $e$-regular partitions   
   \end{itemize}
  It turns out that each set $\Phi_{e,{\bf s}}(n)$ with ${\bf s}\in {\mathfrak{s}}$ gives a natural labelling for the irreducible representations of the Ariki-Koike algebra $\mathcal{H}_n^{\mathfrak{s}} (q)$. As a consequence, 
there are several natural possibilities for the labelling of the simple modules of $\mathcal{H}_n^{\mathfrak{s}} (q)$, one for each choice of an 
 element in the orbit $\mathfrak{s}$. For more details on these parametrizations, we refer to \cite{GJ}. Thus,  one can write:
 $$\operatorname{Irr} ( \mathcal{H}_n^{\mathfrak{s}} (q) )=\{ D^{\ulambda}_{\bf s} \ |\ \ulambda \in \Phi_{e,{\bf s}} (n)\}.$$
 By \cite{B}, each of these labellings has an interpretation in terms of a cellular structure.  
Last,  clearly, if  ${\bf s}$ and ${\bf s}'$ in the same orbit, 
there is a bijection:
$$\Psi_e^{{\bf s} \to {\bf s}'} : \Phi_{(e,{\bf s})} (n)  \to \Phi_{(e,{\bf s}')} (n),$$
which is uniquely defined as follows. For all $\ulambda \in \Phi_{(e,{\bf s})} (n)$ then:
$$D^{\ulambda}_{\bf s} \simeq D^{\Psi_e^{{\bf s} \to {\bf s}'}  (\ulambda)}_{{\bf s}'}.$$
This bijection has been explicitly described in \cite{JL} in a combinatorial way using crystal isomorphisms (the coincidence of  the crystal isomorphisms with  these bijections is proved in  \cite[Prop. 3.7]{Jone}.)  We recall this description subsection (a program in GAP3 is available for computing it in all cases \cite{algo}). In the next sections,  the following particular case: 
 ${\bf s}=(s_1,s_2)$ and ${\bf s}'=(s_1,s_2+e)$ will be of particular interest. 
 
 \begin{Rem}
 If ${\bf s}'$ and ${\bf s}''$ are both very dominant multicharges in the same orbit then $\Psi_e^{{\bf s} \to {\bf s}'}$ is the identity.

 \end{Rem}

 \begin{exa}\label{exa}
 Assume that $e=3$. Take $\mathfrak{s}=(0+3\mathbb{Z},1+3\mathbb{Z})$. Take $n=3$, then, we have
 $$\Phi_{3,(0,1)}(3)=\{ (\emptyset,(3)), ((1),(1,1)), ((1),(2)), ((2),(1)),  ((2,1),\emptyset), ((3),\emptyset) \}$$
 $$\Phi_{3,(0,4)}(3)=\{ (\emptyset,(3)), ((1),(1,1(), ((1),(2)), ((2),(1)),  (\emptyset,(2,1)), ((1,1),(1))\}$$
  $$\Phi_{3,(1,0)}(3)=\{ ((3),\emptyset), ((1),(1,1)), ((1),(2)), ((1,1),(1)),  ((2,1),\emptyset), ((2),(1))\}=\Phi_{3,(4,0)}(3)$$
So that :
    $$\begin{array}{rcl}
    \operatorname{Irr} ( \mathcal{H}_n^{\mathfrak{s}} (q) )&=&\{ D_{(0,1)}^{(\emptyset,3)}   , D_{(0,1)}^{((1),(1,1))} ,D_{(0,1)}^{((1),(2))},D_{(0,1)}^{((2),(1))},D_{(0,1)}^{((2,1),\emptyset)} ,  D_{(0,1)}^{((3),\emptyset)}     \}\\
    &=&\{ D_{(0,4)}^{(\emptyset,(3))}   , D_{(0,4)}^{((1),(1,1))} ,D_{(0,4)}^{((1),(2))},D_{(0,4)}^{((2),(1))},D_{(0,4)}^{(\emptyset,(2,1))} ,  D_{(0,4)}^{((1),(1,1))} \}\\    
        &=&\{ D_{(1,0)}^{(3,\emptyset)}   , D_{(1,0)}^{((1),(1,1))} ,D_{(1,0)}^{((1),(2))},D_{(1,0)}^{((1,1),(1))},D_{(1,0)}^{((2,1),\emptyset)} ,  D_{(1,0)}^{((2),(1))} \}\\    
    \end{array}
    $$

 \end{exa}

\subsection{Description of the crystal isomorphisms}

First let us assume that $l=2$ and that   $(s_1,s_2)\in \mathbb{Z}^2$. Let $\ulambda \in \Phi_{(e,{\bf s})} (n)$.  We follow the presentation in  \cite{JL}. 

 We define 
the minimal integer $d \geq \vert s_1 - s_2\vert$ such that $\lambda^1_{d+1+s_1-s_2} = \lambda^2_{d + 1} = 0$ if $s_2 \geq s_1$, and otherwise the minimal integer $d\geq \vert s_1 - s_2\vert$ such that $\lambda^2_{d+1+s_2-s_1} = \lambda^1_{d + 1} = 0$.
 To $(\lambda^{1},\lambda^{2})$, we associate its
${\bf s}$-symbol of length $d$.  This is the following two-rows array. 
\begin{itemize}
\item If $s_1\leq s_2$ then:
$$S (\lambda^{1},\lambda^2)=\left(
\begin{array}{lllll}
s_{2}-d+\lambda_{d}^{2} & \ldots & \ldots & s_{2}-2+\lambda_{2}^{2} & s_{2}+\lambda_{1}^{2}-1 \\
s_{2}-d+\lambda_{d+s_{1}-s_{2}}^{1} & \ldots & s_{1}+\lambda_{1}^{1}-1
\end{array}
\right)$$
\item if $s_1> s_2$ then:
$$S (\lambda^{1},\lambda^2)=\left(
\begin{array}{lllll}
s_{1}-d+\lambda_{d+s_{2}-s_{1}}^{2} & \ldots & s_{2}+\lambda_{1}^{2}-1\\
s_{1}-d+\lambda_{d}^{1} & \ldots & \ldots & s_{1}-2+\lambda_{2}^{1} & s_{1}+\lambda_{1}^{1}-1 
\end{array}
\right)$$

\end{itemize}

 We will write $S(\lambda^{1},\lambda^{2})=\binom{L_{2}}{L_{1}}$ where the
top row (resp. the bottom row) corresponds to $\lambda^{2}$ (resp.
$\lambda^{1}$).  Of course, it is easy to recover the $2$-partition from the datum of its symbol. 
From this symbol, we define a new symbol 
 $\binom{\widetilde{L}_{2}}{\widetilde{L}_{1}}$ as follows.

\begin{itemize}
\item 
Suppose first $s_{2}\geq s_{1}.\;$Consider $x_{1}=\min\{t\in
L_{1}\}.\;$We associate to $x_{1}$ the integer $y_{1}\in L_{2}$ such that
\begin{equation}
y_{1}=\left\{
\begin{array}
[c]{l}%
\max\{z\in L_{2}\mid z\leq x_{1}\}\text{ if }\min\{z\in L_{2}\}\leq x_{1},\\
\max\{z\in L_{2}\}\text{ otherwise.}%
\end{array}
\right.  \label{algo1}%
\end{equation}
We repeat the same procedure to the lines $L_{2}-\{y_{1}\}$ and $L_{1}%
-\{x_{1}\}.$ By induction this yields a sequence $\{y_{1},...,y_{d+s_{1}%
-s_{2}}\}\subset L_{2}.\;$Then we define $\widetilde{L}_{2}$ as the line
obtained by reordering the integers of $\{y_{1},...,y_{d+s_{2}-s_{1}}\}$ and
$\widetilde{L}_{1}$ as the line obtained by reordering the integers of
$L_{2}-\{y_{1},...,y_{d+s_{1}-s_{2}}\}+L_{1}$ (i.e. by reordering the set
obtained by replacing in $L_{2}$ the entries $y_{1},...,y_{d+s_{1}-s_{2}}$
by those of $L_{1}$).  We obtain a ``symbol''   $\binom{\widetilde{L}_{2}}{\widetilde{L}_{1}}$.

\item  Now, suppose $s_{2}<s_{1}.\;$Consider $x_{1}=\min\{t\in L_{2}%
\}.\;$We associate to $x_{1}$ the integer $y_{1}\in L_{1}$ such that
\begin{equation}
y_{1}=\left\{
\begin{array}
[c]{l}%
\min\{z\in L_{1}\mid x_{1}\leq z\}\text{ if }\max\{z\in L_{1}\}\geq x_{1},\\
\min\{z\in L_{1}\}\text{ otherwise.}%
\end{array}
\right.  \label{algo2}%
\end{equation}
We repeat the same procedure to the lines $L_{1}-\{y_{1}\}$ and $L_{2}%
-\{x_{1}\}$ and obtain a sequence $\{y_{1},...,y_{d+s_{1}-s_{2}}\}\subset
L_{1}.\;$Then we define $\widetilde{L}_{1}$ as the line obtained by reordering
the integers of $\{y_{1},...,y_{d+s_{2}-s_{1}}\}$ and $\widetilde{L}_{2}$ as
the line obtained by reordering the integers of $L_{1}-\{y_{1}%
,...,y_{d+s_{2}-s_{1}}\}+L_{2}.$ 
 We obtain a ``symbol''  $\binom{\widetilde{L}_{2}}{\widetilde{L}_{1}}$. 
\end{itemize} 
The new symbol  $\binom{\widetilde{L}_{2}}{\widetilde{L}_{1}}$ that we obtain is canonically associated to a bipartition $(\overline{\lambda}^1,\overline{\lambda}^2)$ and the multicharge $(s_2,s_1)$. The crystal isomorphisms
 in the case $l=2$ are thus entirely determined from the following results proved in \cite{JL}:
 \begin{enumerate}
 \item We have $\Psi_e^{(s_1,s_2) \to \sigma_1(s_1,s_2)} (\lambda^1,\lambda^2)=(\overline{\lambda}^1,\overline{\lambda}^2)$. 
 \item We have $\Psi_e^{(s_1,s_2) \to \tau.(s_1,s_2)} (\lambda^1,\lambda^2)=(\lambda^2,\lambda^1)$. 
 \item For all $\sigma=x_1.\ldots.x_m\in  \widehat{\mathfrak{S}}_2$ with $x_i \in \{\sigma_1,\tau\}$ for all $i=1,\ldots,m$, we have:
 $$\Psi_e^{(s_1,s_2) \to \sigma.(s_1,s_2)} =\Psi_e^{x_2.\ldots.x_{m} .(s_1,s_2) \to \sigma .(s_1,s_2)} \circ \ldots \circ \Psi_e^{(s_1,s_2) \to x_m. (s_1,s_2)} $$
 \end{enumerate}
 In the general case $l\in \mathbb{N}_{>0}$ and ${\bf s}\in \mathbb{Z}^l$, now:
  \begin{enumerate}
 \item  For all $c=1,\ldots,l-1$, we have  $\Psi_e^{(s_1,s_2) \to \sigma_c(s_1,s_2)} (\ulambda)=\umu$,
  where $\mu^j=\lambda^j$ for all $j\neq c,c+1$, $\mu^c=\overline{\lambda}^c$ and  $\mu^{c+1}=\overline{\lambda}^{c+1}$.
 \item We have $\Psi_e^{{\bf s} \to \tau. {\bf s}} (\ulambda)=(\lambda^2,\ldots,\lambda^l,\lambda^1)$. 
 \item For all $\sigma=x_1.\ldots.x_m\in  \widehat{\mathfrak{S}}_2$ with $x_i \in \{\sigma_1,\ldots,\sigma_{l-1},\tau\}$ for all $i=1,\ldots,m$, we have:
 $$\Psi_e^{{\bf s} \to \sigma.{\bf s}} =\Psi_e^{x_2.\ldots.x_{m} .{\bf s} \to \sigma. {\bf s}} \circ \ldots \circ \Psi_e^{{\bf s} \to x_m .{\bf s}} $$
 \end{enumerate}

\begin{exa}\label{exacomp}
Assume that $(s_1,s_2)\in \mathbb{Z}^2$ with $s_1\leq s_2$. In the next sections, we will be particularly interested in the computation of 
 $\Psi_e^{(s_1,s_2) \to (s_1,s_2+e)}$. Let $\ulambda=(\lambda^1,\lambda^2)\in \Phi_{(e,{\bf s})} (n)$, we then write its symbol:
 $$S (\lambda^{1},\lambda^2)=\left(
\begin{array}{lllll}
s_{2}-d+\lambda_{d}^{2} & \ldots & \ldots & s_{2}-2+\lambda_{2}^{2} & s_{2}+\lambda_{1}^{2}-1 \\
s_{2}-d+\lambda_{d+s_{1}-s_{2}}^{1} & \ldots & s_{1}+\lambda_{1}^{1}-1
\end{array}
\right).$$
We then perform the above algorithm to obtain a new symbol $\binom{\widetilde{L}_{2}}{\widetilde{L}_{1}}$  which must be of the form :
$$\left(\begin{array}{llllll}
y_{d+s_{1}-s_{2}} & \ldots & y_{1} \\
x_{d} & \ldots & \ldots & x_{2} & x_{1} 
\end{array}
\right)$$
We then consider the following symbol:
$$\left(\begin{array}{llllllll}
0 & \ldots & e-1 & x_{d} +e & \ldots & \ldots & x_{2} +e & x_{1} + e \\
y_{d+s_{1}-s_{2}} & \ldots & y_{1} \\
\end{array}
\right)$$
By the discussion above, this is the $(s_1,s_2+e)$-symbol of the bipartition  $\Psi_e^{(s_1,s_2) \to (s_1,s_2+e)}(\lambda^1,\lambda^2)$ (more details and examples can be found in \cite{JB})
\end{exa}
 \begin{exa}\label{exa2} We keep the example \ref{exa}, one can check that the map $\Psi_e^{(0,1) \to (0,4)}$ is given as follows
$$
\begin{array}{cccc}
\Psi_e^{(0,1) \to (0,4)}:& \Phi_{3,(0,1)} & \to & \Phi_{3,(0,4)}\\
 &  (\emptyset,(3)) & \mapsto &  (\emptyset,(3)) \\
 &  ((1),(1,1)) & \mapsto &  ((1),(1,1)) \\
 &  ((1),(2)) & \mapsto &  (\emptyset,(2,1)) \\
 &  ((2),(1)) & \mapsto &  ((2),(1)) \\
 &  ((2,1),\emptyset) & \mapsto &  ((1,1),(1)) \\
  &  ((3),\emptyset) & \mapsto &  ((1),(2)) \\
\end{array}$$
More examples can be found  in \cite{JL}. 

 \end{exa}

\subsection{Aperiodic multisegments and multipartitions}\label{multis}

Let $\mathfrak{s}$ be an orbit of $\mathbb{Z}^l$ with respect to the action of the affine symmetric group (recall the definition of the action in \S \ref{act}). If $V$ is a simple module 
 for the Ariki-Koike algebra  then it is also a simple  $\widehat{H}_n (q)$-module in the category  $\operatorname{Mod}_n$. Hence there exists a unique 
  aperiodic multisegment $\psi$ such that $V \simeq L_{\psi}$ (as a $\widehat{H}_n (q)$-module).  As a consequence, far any ${\bf s} \in\mathfrak{s}$
  we have a well defined map:
  $$\chi^n_{e,{\bf s}} : \Phi_{(e,{\bf s})} (n) \to \mathfrak{M}_e (n),$$
  which is defined as follows. Let $\ulambda \in \Phi_{(e,{\bf s})} (n)$, then we have a unique $\chi^n_{(e,{\bf s})} (\ulambda)\in \mathfrak{M}_e (n)$ such that:
  $$D^{\ulambda}_{\bf s} \simeq L_{\chi^n_{e,{\bf s}} (\ulambda)}.$$
 By \cite{AJL}, this map may be described as follows:
\begin{itemize}
\item   Assume first that ${\bf s}\in \mathcal{A}_e^l[s]$ for all non zero part $\lambda^c_i$ of $\ulambda$, we associate the segment 
$$[(1-i+s_c)+e\mathbb{Z},\ldots,\lambda^c_i-i+s_c].$$
By \cite{AJL},  The multisegment 
$\chi^n_{e,{\bf s}} (\ulambda)$ is just the formal sum of all the segments associated to the non zero part of $\ulambda$.
\item  As a consequence, in general, if ${\bf s}'\in \mathfrak{s}$. Let ${\bf s} \in \mathcal{A}_e^l[s]\cap \mathfrak{s}$, then
$$\chi^n_{e,{\bf s}'} (\ulambda)=\chi^n_{e,{\bf s}} (\Psi_e^{{\bf s}'\to {\bf s}}  (\ulambda)).$$
\end{itemize}
Given an aperiodic  multisegment $\psi$, 
It is now natural to try to find the multicharges ${\bf s} $  such that $\psi$ as an antecedent for the map $\chi^n_{e,{\bf s}}$. This question has been completely  solved in \cite{JL3}.  There always exist such multicharges (they are non unique in general)  which  are called {\it admissible multicharges}.   By \cite{AJL},  $\chi^n_{e,{\bf s}} $ is injective so that if ${\bf s}$ is admissible for $\psi$  there exists a unique 
  $\ulambda$ such that $\chi^n_{e,{\bf s}} (\ulambda)=\psi$. This $l$-partition will be called {\it admissible} (with respect to $\psi$).  By definition, we have the following proposition where we use the following notation. For ${\bf s}$ and ${\bf t}$ two multicharges, we denote ${\bf s} \subset {\bf t}$ if and only if, for all $j\in \mathbb{Z}/e\mathbb{Z}$, the number 
   of integers conguent to $j$ in ${\bf s}$ is less or equal to the number 
   of integers conguent to $j$ in ${\bf t}$.
   
     \begin{Prop}\label{adm}
  Assume that $\ulambda\in  \Phi_{(e,{\bf s})} (n)$  then ${\bf t}$ is admissible for the multisegment 
   $\chi^n_{e,{\bf s}} (\ulambda)$ if and only if ${\bf s} \subset {\bf t}$. 
  \end{Prop}
  \begin{proof}
  Set ${\bf s}=(s_1,\ldots,s_l)$ and ${\bf t}=(t_1,\ldots,t_m)$. 
  Assume that $\ulambda\in  \Phi_{(e,{\bf s})} (n)$ then as a $\widehat{H}_n (q)$-module, we have that  $\prod_{1\leq j\leq l} (X_1-q^{s_j}) $ acts as $0$ on 
   $D^{\ulambda}_{\bf s} \simeq L_{\chi^n_{e,{\bf s}} (\ulambda)}$. As a consequence, as  ${\bf s} \subset {\bf t}$, we have that 
     $\prod_{1\leq j\leq m} (X_1-q^{t_j}) $ acts as $0$ on 
   $L_{\chi^n_{e,{\bf s}} (\ulambda)}$. This implies that it is a well-defined  $\mathcal{H}_n^{\mathfrak{t}} (q)$-module and the result follows. 

  \end{proof}
  \begin{Rem}
  One can also prove the above proposition combinatorially using the descriptions of the admissible multicharges. 
  
  \end{Rem}

%
%

\section{The Mullineux and the Iwahori-Matsumoto involutions}

The aim of this section is to introduce the Mullineux involution for the symmetric group and its analogues in the context of  Ariki-Koike algebras and  affine Hecke algebras.

\subsection{Iwahori-Matsumoto involution for affine Hecke algebras of type $A$}

We have an involution $\sharp$  on  $\widehat{H}_n (q)$   which has been defined by
 Iwahori and Mastumoto  in \cite{IM}: 
 $$T_i^{\sharp}=-q T_{i}^{-1},\ X_j^{\sharp}=X_j^{-1}$$
 for $i=1,\ldots,n-1$ and $j=1,\ldots,n$. 
The Iwahori-Matsumoto involution naturally  induces an involution on the set of aperiodic multisegments.  We have an involution:
$$\sharp : \mathfrak{M}^e (n) \to \mathfrak{M}^e (n),$$
defined for all $\psi \in \mathfrak{M}^e (n)$ by 
$$L_{\psi}^{\sharp}=L_{\psi^\sharp}.$$

\begin{Rem}

 We have in fact  two others  well defined involutions on $\widehat{H}_n (q)$ which are defined as follows:
\begin{itemize}
\item The Zelevinsky involution $\tau$ defined in \cite{MW} : 
 $$T_i^{\sharp}=-q T_{n-i}^{-1},\ X_j^{\sharp}=X_{n+1-j}^{-1},$$
 for $i=1,\ldots,n-1$ and $j=1,\ldots,n$. 
\item The  involution $\nabla$ : 
 $$T_i^{\nabla}=-q T_{n-i},\ X_j^{\nabla}=X_{n+1-j},$$
 for $i=1,\ldots,n-1$ and $j=1,\ldots,n$. 
\end{itemize}
We have for all $x\in \widehat{H}_n (q)$:
$$x^{\tau}=(x^{\nabla})^{\sharp}=(x^{\sharp})^{\nabla}.$$
These two involutions thus also induce involutions on the set $\mathfrak{M}^e (n)$ and they  have been studied in \cite{JL3}. 
 \end{Rem}

\subsection{Mullineux involution for Ariki-Koike algebras}
Assume that ${\bf s}\in \mathbb{Z}^l$.  Then we have a well-defined algebra automorphism:
$$\gamma:\mathcal{H}_n^{\mathfrak{s}} (q) \to \mathcal{H}_n^{\mathfrak{s}} (q^{-1}),$$
which is defined on the generators as follows:
$$T_0 \mapsto T_0^{-1},\ T_i \mapsto -qT_i^{-1}.$$
This map naturally induces  bijections  on the  indexing sets of the simple modules   of Ariki-Koike algebras.   
Let 
$\mathfrak{s}^{\sharp}$ be the orbit of $(-s_1,\ldots,-s_l)$ modulo the action of the affine symmetric group. 
Let ${\bf v} \in \mathfrak{s}^{\sharp}$ then we have a map:
$$m^{{\bf s} \to {\bf v}}_e: \Phi_{(e,{\bf s})} (n) \to \Phi_{(e,{\bf v})} (n),$$
defined as follows. Let $\ulambda \in \Phi_{(e,{\bf s})} (n)$, then there exists a unique  $\umu\in \Phi_{(e,{\bf v})} (n)$
such that 
$$(D_{{\bf s}}^{\ulambda})^\gamma \simeq D_{{\bf v}}^{\umu},$$
and we set 
$$m^{{\bf s} \to {\bf v}}_e (\ulambda)=\umu.$$
This map 
has been described in \cite{JL}.  If $l=1$ and $e$ is prime then it coincides with the usual Mullineux involution 
 of the symmetric group that we have defined in the introduction.  If $l=1$, then it corresponds to the Mullineux involution 
  of the Hecke algebra of type $A$ of \cite{Br} which will simply be denoted by $m_e$ (it does not depend on ${\bf s}$). In this paper, we will give an algorithm for computing $m_e$. 
  
  \begin{Rem}\label{core}
  If $\lambda$ is a partition and $\gamma$ a node of its Young diagram,  the $\gamma$-hook of $\lambda$ id by definition the set of  all the nodes at the right and at the bottom of $\gamma$ (including $\gamma$). The length of the hook is the number of nodes in it. We say that $\lambda$ is an $e$-core if all the hooks have length strictly less than $e$. If $\lambda$ is an $e$-core then $m_e (\lambda)$ can be easily described: it is just the conjugation of $\lambda$ (as in the semisimple case), see \cite{Mu} (when $e$ is a prime but the results generalizes easily if $e$ is an integer). 
  
  \end{Rem}

  More generally, it is a natural question to ask how one can 
 describe all the maps $m^{{\bf s} \to {\bf v}}_e $  
  in general. It turns out that by \cite[Prop. 4.2]{JL0}, knowing the map $m_e$,  one can describe it quite easily in a particular case: 
\begin{Prop}\label{vd}
Assume that ${\bf s}$ is very dominant. Let  ${\bf s}^{\sharp}:=(-s_1',\ldots,-s_l')$ be a very dominant multicharge such that 
 $s_i'\equiv s_i +e\mathbb{Z}$  for all $i=1,\ldots,l$. 
Then for all 
$\ulambda \in  \Phi_{(e,\mathfrak{s})} (n)$, we have:
$$m^{{\bf s} \to {\bf s}^{\sharp}}_e (\ulambda) =(m_e(\lambda^1),\ldots,  m_e (\lambda^l)).$$
\end{Prop}
As a consequence, this result, combining with the fact that we know how to compute the natural bijection between the various parametrizations of the simple modules of Ariki-Koike algebras permit to describe all the Mullineux involutions (assuming that we know $m_e$). Indeed, let ${\bf v}_1 \in \mathfrak{s}$ and let ${\bf v}_2 \in \mathfrak{s}^{\sharp}$. Let 
 ${\bf s}_1 \in \mathfrak{s}$ be a very dominant multicharge. Then we have:
 $$m^{{\bf v}_1 \to {\bf v}_2}_e=     \Psi_e^{{\bf s}_1^{\sharp}\to {\bf v}_2} \circ m^{{\bf s}_1 \to {\bf s}_1^{\sharp}} \circ  \Psi_e^{{\bf v}_1\to {\bf s}_1}$$
 where ${\bf s}_1^{\sharp}$ is as in the above proposition.

\begin{exa}
We keep the setting of example \ref{exa}. For $n=3$, the multicharge $(0,4)$ is very dominant, so the above result applies in this case. One can take ${\bf s}^{\sharp}=(0,5)$
 which is also very dominant. 
Using the fact that $m_3 (3)=(2,1)$, $m_3 (1.1)=(2)$, we obtain 
$$
\begin{array}{cccc}
m_e^{(0,4) \to (0,5)}& \Phi_{3,(0,1)} (3) & \to & \Phi_{3,(0,5)} (3)\\
 &  (\emptyset,(3)) & \mapsto &  (\emptyset,(2,1)) \\
 &  ((1),(1,1)) & \mapsto &  ((1),(2)) \\
 &  (\emptyset,(2,1)) & \mapsto &  (\emptyset,(3)) \\
 &  ((2),(1)) & \mapsto &  ((1,1),(1)) \\
 &  ((1,1),(1))& \mapsto &  ((2),(1)) \\
  &  ((1),(2)) & \mapsto &  ((1),(1,1)) \\
\end{array}$$
Now combining with our cristal isomorphism in Example \ref{exa2}, we for example obtain 
$$
\begin{array}{cccc}
m_e^{(0,1) \to (0,5)} : & \Phi_{3,(0,1)} & \to & \Phi_{3,(0,5)}\\
 &  (\emptyset,(3)) & \mapsto &  (\emptyset,(2,1)) \\
 &  ((1),(1,1)) & \mapsto &  ((1),(2)) \\
 &  ((1),(2)) & \mapsto &  (\emptyset,(3)) \\
 &  ((2),(1)) & \mapsto &  ((1,1),(1)) \\
 &  ((1,1),(1))& \mapsto &  ((2),(1)) \\
  &  ((3),\emptyset) & \mapsto &  ((1),(1,1)) \\
\end{array}$$

\end{exa}

\subsection{Relations between the involutions}\label{sub}

Now we put all the above results together to deduce relations between the various involutions we have defined. The following result is proved  in \cite{JL3}.

\begin{Th}
Let $\psi$ be an aperiodic multisegment and let ${\bf s}\in \mathcal{A}_e^l[s]$ be an admissible multicharge for $\psi$. 
 Set ${\bf s}^t=(-s_l,\ldots,-s_1)\in \mathcal{A}_e^l[-s_l]$ then we have:
$$\Psi^{\sharp}=\chi^n_{e,{\bf s}^t} \circ m_e^{{\bf s}\to {\bf s}^t} \circ  (\chi^n_{e,{\bf s}})^{-1} (\psi)$$

\end{Th}

As a consequence, the Iwahori-Mastumoto involution may be computed as follows. Take an aperiodic multisegment $\psi$.
\begin{itemize}
\item Choose an admissible multicharge ${\bf s}$ for $\psi$ and compute $\ulambda:=(\chi^n_{e,{\bf s}})^{-1} (\psi)$ using \S \ref{multis}. 
\item Compute $\nu:=m_e^{{\bf s}\to {\bf s}^t}  (\ulambda)$ using the discussion in the last section. 
\item Compute $\psi^{\sharp}:=\chi^n_{e,{\bf s}^t}  (\nu)$ using the algorithm described in \cite{JL3}. 
\end{itemize}

\begin{exa}
Take $e=3$ and  the multisegment $[0]+[0,1,2]+[1,2,3]$. One can see that $(0,1)$ is admissible for this multisegment and we have 
 $(\chi^7_{3,(0,1)})^{-1} (\psi)=((3),(3,1))$.

  We need to compute  $m_e^{{\bf s}\to {\bf s}^t} ((3),(3,1))$. To do this, we first compute $\Psi_e^{(0,1) \to (0,7)}((3),(3,1))$ 
 as $(0,7)$ is very dominant. We obtain the bipartition  $((1),(3,3))$. Now we have seen that 
 $$m^{(0,7) \to (0,8)}_e ((1),(3,3))=(m_3 (1),m_3 (3,3))=((1),(6)).$$
 Again, we compute $\Psi_e^{(0,8) \to (0,2)}((1),(6))=((1),(6))$
 and thus we get 
  $$\psi^{\sharp}:=[0]+[2,0,1,2,0,1]. $$
\end{exa}

 Now, let us explain how  one can deduce an algorithm for computing the Mullineux involution for $e$-regular partitions. 
This  is based on the following elementary remark. Let $\lambda \in \Phi_{e,(0)}$ be an 
 $e$-regular partition and consider the aperiodic multisegment $\psi:=\chi^n_{e,(0)} (\lambda)$ (recall that this is nothing but the formal sum of the segments given by the rows of the Young diagram of $\lambda$). The above theorem shows that:
 $$m_e (\lambda)= (\chi^n_{e,(0)})^{-1} (\psi^{\sharp}).$$
 So now we are reduced to compute  $ (\chi^n_{e,(0)})^{-1} (\psi^{\sharp})$. Take ${\bf s}\in \mathcal{A}_e^l[0]$ such that $l>1$ then by Proposition \ref{adm}, this is an admissible multicharge.  We have:
$$\psi^{\sharp}=\chi^n_{e,{\bf s}^t} \circ m_e^{{\bf s}\to {\bf s}^t} \circ (\chi^n_{e,{\bf s}})^{-1} (\psi)$$
Now $\umu:=(\chi^n_{e,{\bf s}})^{-1} (\psi)$ is the admissible $l$-partition and the main problem is thus to compute $m_e^{{\bf s}\to {\bf s}^t }(\umu)$.
 We have already seen that this can be done in three steps:
 \begin{enumerate}
 \item Compute the  crystal isomorphism $\Psi_e^{{\bf s} \to {\bf v}} (\umu)=(\nu^1,\ldots,\nu^l)$ where ${\bf v}$ is very dominant (recall that this means that ${\bf v}=(s_1,s_2+ke)$ 
  with $ke>n-1$)
 \item By Proposition \ref{vd},  $m_e^{{\bf v}\to {\bf v}^\sharp} (\unu)$ can be computed by applying the Mullineux map component by component. 
  As $|\nu|=|\lambda|$, if we assume that at least  two components of the $l$-partition $(\nu^1,\ldots,\nu^l)$
   are non empty, all of the components are of rank $<n$ and we know how to compute the Mullineux involution by induction. 
   \item Apply again a crystal isomorphism $\Psi_e^{{\bf v}^\sharp \to {\bf s}^t}$.
 \end{enumerate}

In the next section, we will apply the above algorithm in the case where $l=2$ and in particular show that the condition for applying our induction  in step  $2$ 
is always satisfied (except in the case where ${\bf s}=(s_1,s_2)$ and $s_1=s_2$.)

\section{Combinatorial properties}

In this section, we will try to find simple combinatorial ways to compute several objects that we have already defined: 
 this concerns the admissible multicharges and multipartitions and the crystal isomorphisms. 

\subsection{On admissible multipartitions}
If $\lambda$ and $\mu$ are two partitions, we denote by $\lambda \sqcup \nu$ the partition obtained by 
 concatenation (and reordering the parts if necessary). 
 
Assume that we have an  $e$-regular partition $\lambda=(\lambda_1,\ldots,\lambda_r)$ (that is $\lambda \in \Phi_{e,(s)} (n)$ for any $s\in \mathbb{Z}$). Let 
${\bf s} \in \mathcal{A}_e^l[s]$. By Proposition \ref{adm}, ${\bf s}$ is an admissible multicharge.  The aim of this subsection is to show that one can easily  construct the associated admissible $l$-partition $\ulambda \in \Phi_{e,{\bf s}} (n)$ such that 
   $\chi^n_{e,{\bf s}} (\ulambda)=\chi^n_{e,( s)} (\lambda)$ (recall that $\chi^n_{e,{\bf s}}$ is always injective).   To do this, one can use the algorithm developed in \cite{JL3} from the datum 
    of the multisegment  $\chi^n_{e,( s)} (\lambda)$ or we can  argue as follows.  
Let $l'\in \{1,\ldots,l\}$ be minimal such that $s_{l'}=s_l$. We construct $\ulambda$ by induction as follows.   

 If $\lambda=\emptyset$ then $\ulambda :=\uemptyset$ and we are done. Otherwise, 
set 
$${\bf s}':=\left\{ \begin{array}{cl}
  (\underbrace{s_l,\ldots,s_l}_{l-l'+2},s_2+e,\ldots, s_{l'-1}+e) & \text{ if }
l'\neq 1\\
{\bf s} & \text{ if }l'=1
\end{array}
\right.$$
Note that we have  ${\bf s}' \in  \mathcal{A}_e^l[s_l] $. 
 We denote $m:=\lambda_1+\ldots+\lambda_{e+s-s_l}$. 

By induction, we have constructed the $l$-partition $\unu \in \Phi_{(e,{\bf s}')} (n-m)$
 such that   we have 
   $$\chi^{n-m}_{e,{\bf s}'} (\unu)=\chi^{n-m}_{e,(s_l)}  (\lambda_{e+s-s_l+1},\lambda_{e+s-s_l+2},\ldots,\lambda_r)$$
  We  then define $\ulambda$ as follows
  \begin{itemize}
\item If we have $l'=1$ then $\lambda^1=(\lambda_1,\ldots,\lambda_{e}) \sqcup \nu^{l}$ and   $\lambda^j=\nu^{j-1}$ if $j\neq 1$. 
  \item Otherwise, 
   $\lambda^1=(\lambda_1,\ldots,\lambda_{e+s-s_l}) \sqcup \nu^{2+l-l'}$ and $\lambda^j=\nu^{j+1-l'}$ for $j>1$  where the indices are understood modulo $l$.
   \end{itemize}
   \begin{Prop}
      With this construction, we have $\ulambda\in \Phi_{e,{\bf s}} (n)$ and  $\chi^n_{e,{\bf s}} (\ulambda)=\chi^n_{e,( s)} (\lambda)$. 
   \end{Prop}
   \begin{proof}
We prove the proposition by induction. The result is trivial when $n=0$. Keeping the above notations, one can assume that  $\unu \in \Phi_{(e,{\bf s}')} (n-m)$. 
 First one can perform exactly the same procedure as in \S \ref{multis} for the description of the map  $\chi^n_{e,{\bf s}}$
  to associate to $\ulambda$ a multisegment (even if we have - not already - proved that $\ulambda$ is in $\Phi_{e,{\bf s}} (n)$). By construction, this multisegment is nothing but 
 $\chi^n_{e,(s)} (\lambda)$. It is thus an aperiodic multisegment. This proves  condition $3$ of FLOTW $l$-partition for $\ulambda$ (see the definition in \S \ref{act}).  
Hence, we just need to show that the $l$-partition satisfies  the two first points. 
\begin{itemize}
\item If $l'=1$, by induction, we have $\nu^j\geq \nu^{j+1}$ for all $j=1,\ldots,l-1$. This implies that $\lambda_i^j\geq \lambda_i^{j+1}$ for all $j=2,\ldots,l-1$
 and that $\lambda^l_i\geq \lambda^1_{i+e}$ for all $i\geq 1$ and we get that $\lambda^1_i\geq \lambda^2_i$ because $(\lambda_1,\ldots,\lambda_{e})$ are the greatest parts of $\lambda$ and because $\nu^l_i\geq \nu^1_{i+e}$ for all $i>0$. 
\item If $l'\neq 1$,     by the property of FLOTW $l$-partitions, we have that $\umu:=(\nu^{l-l'+3},\ldots,\nu^l,\nu^1,\ldots,\nu^{l-l'+1},\nu^{l-l'+2})$ is in 
      $\Phi_{e,{\bf v}} (n-m)$ for ${\bf v}=(s_2,\ldots,s_{l'-1},s_l,\ldots,s_l,s_l)$ and we can thus conclude using the fact that 
       $\lambda^1_j=\lambda_j$ if $j=1,\ldots, e+s-s_l$ and $\lambda^1_j=\mu^l_{j-(e+s-s_l)}$ otherwise.

\end{itemize}

   \end{proof}
   
In the case where $l=2$ (which is the case that we will mostly studied in the forthcoming sections), the multipartition $\ulambda=(\lambda^1,\lambda^2)$ is easy to obtain. One can assume that $s_1=0$, then  we have 
$$\lambda^1=(\lambda_1,\ldots,\lambda_{e-s_2},\lambda_{2e-s_2+1},\ldots,\lambda_{3e-s_2},\ldots,\lambda_{2ke-s_2+1},\ldots,\lambda_{3ke-s_2},\ldots)$$
 and 
 $$\lambda^2=(\lambda_{e-s_2+1}\ldots,\lambda_{2e-s_2},\lambda_{3e-s_2+1},\ldots,\lambda_{4e-s_2+1},\ldots\,\lambda_{3ke-s_2+1},\ldots,\lambda_{4ke-s_2},\ldots)$$
 \begin{exa}
 Let us take $e=4$, $\lambda=(8,8,6,6,4,3,3,2,1,1)$, then the associated Young tableau (with the residues of each node marked in the associated box)  is:
 
   $$
\begin{array}{|c|c|c|c|c|c|c|c|}
  \hline
  0&1  &2 & 3 &  0&1&2&3  \\
  \cline{1-8}
  3&0 &1 & 2 &  3&0&1&2 \\
  \cline{1-8}
2 &  3&0&1&2& 3 \\
      \cline{1-6}
 1 & 2 &  3&0&1&2 \\     
       \cline{1-6}
 0&1&2& 3 \\
  \cline{1-4} 
  3&0&1\\
  \cline{1-3}   
    2&3&0\\
      \cline{1-3}   
      1 & 2   \\
  \cline{1-2}
      0 \\
    \cline{1-1}   
    1  \\
        \cline{1-1} 
\end{array}$$
Take ${\bf s}=(0,2,2)$. Following the algorithm, we first have $l'=2$. Then ${\bf s}'=(2,2,2)$. We have $m=\lambda_1+\lambda_2$
 and we need to compute $\unu$ such that  
     $$\chi^{n-m}_{4,(2,2,2)} (\unu)=\chi^{n-m}_{4,(2)}  (6,6,4,3,3,2,1,1)$$
We obtain $\unu=((6,6,4,3),(3,2,1,1),\emptyset)$ and we have $\ulambda=((8,8),(6,6,4,3),(3,2,1,1))$.

In the case where $l=2$, we have:
\begin{itemize}
\item If ${\bf s}=(0,0)$, we have $\ulambda=((8,8,6,6,1,1),(4,3,3,2))$.

\item If ${\bf s}=(0,1)$, we have $\ulambda=(8,8,6,2,1,1),(6,4,3,3))$.

\item If ${\bf s}=(0,2)$, we have $\ulambda=(8,8,3,2,1,1),(6,6,4,3))$.

\item If ${\bf s}=(0,3)$, we have $\ulambda=((8,3,3,2,1),(8,6,6,4,1))$.

\end{itemize}

 \end{exa}
 
 Using this, we have thus constructed a map 
 $$\theta^n_{e,{\bf s}} : \Phi_{e,(0)} (n)\to \Phi_{e,{\bf s}} (n)$$
 which associates to $\lambda$ the $l$-partition $\ulambda$ constructed above (we will sometimes omit the subscript $n$).

\subsection{Crystal isomorphisms}\label{crysi}

In this second subsection, we study in details  the crystal isomorphisms 
 restricted to the multipartitions in  the image of $\theta_{e,{\bf s}}$.
in the case where $l=2$.  The first aim is to implify the procedure to compute it, the second is to  show certain crucial properties which will show that our algorithm run.

Let $\lambda$ be an $e$-regular partition and assume that ${\bf s}=(0,s)$. We also assume that $\lambda$ is non empty and that $r$ is maximal such that 
 $\lambda_r \neq 0$. 
 Let $(\lambda^1,\lambda^2):=\theta_{(e,(0,s))} (\lambda)$ and consider the associated symbol with length $te$ with $t$ sufficiently large. 
It is thus of the following form :
$$\left(
\begin{array}{cccccccccccc}
\alpha_{te} & \ldots & \alpha_{(t-1)e+1} & \ldots&\alpha_{2e} & \ldots &  \alpha_{e+1}&  \alpha_e &\ldots&  \alpha_{s+1}&\ldots&\alpha_1 \\
 \beta_{te-s} & \ldots & \alpha_{(t-1)e-s+1}   &\ldots&\beta_{2e-s} &   \ldots & \beta_{e-s+1} &  \beta_{e-s}&\ldots &        \beta_1 &
\end{array}
\right)$$
By definition of the symbol, we here have $\alpha_j:=\lambda^2_j-j+s$ for $j=1,\ldots,ke$ and $\beta_j:=\lambda^2_j-j$  for $j=1,\ldots,ke-s$. We denote $(\mu^1,\mu^2):=\Psi^{(0,s)\to (0,s+k.e)}_e (\lambda^1,\lambda^2)$ (so that, as usual, $ke>n-1$ and thus so that the multicharge $(0,s+ke)$ is very dominant)

\noindent \underline{Assume that $\lambda \neq \emptyset$ and that $\mu^2=\emptyset$} then  the algorithm for the computation of 
$ \Psi^{(0,s)\to (0,s+k.e)}_e$ easily shows that  that this can happen if and only if $ \Psi^{(0,s)\to (0,s+k.e)}_e$ is the identity. 
 This  thus  implies that 
$$\{\beta_i\ |\ i=1,\ldots,ke-s\} \subset \{\alpha_i\ |\ i=1,\ldots,ke\}$$
 In this case, we also need to have 
$r\leq e-s$.
Now we have for all  $i=1,\ldots,ke$, $\alpha_i=-i+s$ and  also  $\beta_{j}\leq \alpha_j$ for all $j=1,\ldots,ke-s$. As a consequence, we have
$$\lambda^2_1-1\leq -1+s$$
and thus $\lambda^2_2\leq s$. We conclude 
\begin{Prop}\label{diff0}
Under the above notations, assume that $\mu^2=\emptyset$ then $\lambda=\lambda^1$ is an $e$-core. 
\end{Prop}
\begin{proof}
The above discussion shows that $\lambda$ has at most $e-s$ non empty rows and at most $s$ columns. This implies that the hooks of 
 $\lambda$ has at most length $e-1$ and thus that $\lambda$ is an $e$-core. 

\end{proof}
\noindent \underline{Now let us see what we can say if $\mu^1=\emptyset$.} Before this, we  show below that the image of $\ulambda$ under a crystal isomorphism can be 
quite easily computed in the case where $\ulambda$ is in the image of $\theta_{e,{\bf s}}$ which is the case we are interested in here. 

Keeping, the above notations, 
for all $i=1,\ldots,k-1$, we have $\alpha_{ie}=\lambda_{2ie-s}-ie+s$ and $\beta_{ie-s+1}=\lambda_{2ie-s+1}-(ie-s+1)$.
So we have $\alpha_{ie}+ie-s\geq \beta_{ie-s+1}+ie-s+1$  and thus $\alpha_{ie}> \beta_{ie-s+1}$. 

In addition $\alpha_{ie+1}=\lambda_{2ie+1-s}-(ie+1)+s$ and  $\beta_{(i+1)e-s}=\lambda_{2ie-s}-(( i+1)e-s)$. So we have 
$\beta_{(i+1)e-s}+((i+1)e-s) \geq \alpha_{ie+1}+(ie+1)-s$. So $\beta_{(i+1)e-s}+e>\alpha_{ie+1}$.

These calculations show that one can perform our crystal isomorphism step by steps in the ``blocks'' of the symbol separated by vertical lines below. First recall in Example \ref{exacomp} how the crystal isomorphisms $ \Psi^{(0,s')\to (0,s'+e)}_e$ can be described. 

$$\left(
\begin{array}{ccc|c|ccc|ccccc}
\alpha_{ke} & \ldots & \alpha_{(k-1)e+1} & \ldots&\alpha_{2e} & \ldots &  \alpha_{e+1}&  \alpha_e &\ldots&  \alpha_{s+1}&\ldots&\alpha_1 \\
 \beta_{ke-s} & \ldots & \beta_{(k-1)e-s+1}   &\ldots&\beta_{2e-s} &   \ldots & \beta_{e-s+1} &  \beta_{e-s}&\ldots &        \beta_1 &
\end{array}
\right)$$
We see that all the calculations in the blocks are trivial except in the rightmost.   
After one step of the crystal isomorphism we get

$$\left(
\begin{array}{ccc|c|ccc|cccccccc}
 0& \ldots & e-1 & \ldots& \beta_{3e-s} +e & \ldots &  \beta_{2e-s+1}+e&  \beta_{2e-s}+e &\ldots& \beta_{e+1}  & \ldots &\beta_{e-s+1}+e&\ldots&\alpha_1' \\
\alpha_{ke} & \ldots & \alpha_{(k-1)e+1}     &\ldots&\alpha_{2e}&   \ldots & \alpha_{e+1} &  \beta_{e-s}'&\ldots &        \beta_1' &
\end{array}
\right)$$

and we see that the properties above are always satisfy. In particular, with the notations above, we have. 
$$\beta_{e-s}'+e> \beta_{e-s+1}+e$$
Now, take the 
right end of our first symbol:
$$\left(
\begin{array}{cccccc}
\alpha_{e} & \ldots & \alpha_{s+1} & \alpha_s & \ldots  & \alpha_1  \\
 \beta_{e-s} & \ldots & \beta_{1}  
\end{array}
\right)$$
We already know that  $\beta_{e-s}+e> \alpha_1$.  Assume that we have $\lambda^1_j\neq 0$ so that  $\beta_j>-j$. Then we claim that this implies that  we have $\beta_j\geq \alpha_{s+j-1}$.  To do this, note that we have:
  $$\beta_j\geq \beta_{j-1}+1 \geq \ldots \geq  \beta_{e-s}+(e-s-j)>\alpha_1-s-j.$$
Now we have $\alpha_1\geq  \alpha_2+1\geq \ldots \geq \alpha_{s+j-1}+(s+j-2)$. So
$$\beta_j>\alpha_{s+j-1}-2$$
The only problem may appear if $\beta_j=\alpha_{s+j-1}-1$ and this implies that all the inequalities above are in fact equalities.
 We thus have:
$$\beta_j= \beta_{j-1}+1= \ldots =  \beta_{e-s}+(e-s-j),$$
and 
 $$\alpha_1=  \alpha_2+1\geq \ldots = \alpha_{s+j-1}+(s+j-2)=\beta_j+s-j-1=\beta_{j-1}+s-j=\ldots=\beta_{e-s}+e-1.$$
 This case implies that we have   an $e$-period  in the sense of \cite[Def. 2.2]{JL2}. Such property is impossible for Uglov $l$-partitions by \cite[Prop. 5.1]{JL2}. 
 
 This discussion implies that, under the notations above, if we have $\beta_j>-j$
  then 
  we must have $\beta_j ' >-j$ so that the associated part of the partition is also non zero. By a direct induction, we thus deduce:
  
\begin{Prop}\label{diff1}
Let $0<s<e$ and let $\lambda$ be an $e$-regular partition and $(\lambda^1,\lambda^2):=\theta_{(e,(0,s))} (\lambda)$.
Assume that $(\mu^1,\mu^2):=\Psi^{(0,s)\to (0,s+k.e)}_e (\lambda^1,\lambda^2)$ for $k>>0$ (so that $(0,s+k.e)$ is very dominant, see \S \ref{act}). Then 
$|\mu^1|\neq 0$. 

\end{Prop}
\begin{Rem}
In the case where $s=0$, the above discussion also  shows that if $(\lambda^1,\lambda^2) :=\theta_{(e,(0,0))} (\lambda)$ then 
$\Psi^{(0,s)\to (0,k.e)}_e  ( \lambda^1,\lambda^2) = (\emptyset,\lambda)$ for $k>>0$. As a consequence, this choice of multicharge  cannot be used to get our recursive algorithm to compute the Mullineux involution because then it would require the computation of $m_e (\lambda)$ ... to compute $m_e (\lambda)$. 
\end{Rem}

\section{The algorithm}\label{stepa}
Let $\lambda=(\lambda_1,\ldots,\lambda_r)$ be an $e$-regular partition of rank $n$. We can now present a recursive algorithm for computing $m_e (\lambda)$.   First by Remark \ref{core}, one can assume that $\lambda$ is not an $e$-core. 
The algorithm now  consists in the following steps:
\begin{enumerate}
\item Choose $0<s<e$ and consider the bipartition $(\lambda^1,\lambda^2):=\theta_{(e,(0,s))} (\ulambda)$.
\item Compute  $(\mu^1,\mu^2):=\Psi^{(0,s)\to (0,s+k.e)}_e (\lambda^1,\lambda^2)$ for $k>>0$.  By Propositions \ref{diff0} and \ref{diff1}, we now that 
$|\mu^1|<n$ and $|\mu^2|<n$.  

\item By induction, we know $m_e (\mu^1)$ and $m_e (\mu^2)$ and we can thus compute:
$$(\kappa^1,\kappa^2):=\Psi_e^{(0,-s+ke) \to  ( (0,e-s)} (m_e (\mu^1),m_e(\mu^2)).$$
\item  We have $m_e (\lambda)=\theta_{(e,(0,e-s))}^{-1}(\kappa^1,\kappa^2)$. 
\end{enumerate}
Note that in principle, one can choose an arbitrary  multicharge ${\bf s}$ instead of $(0,s)$ (as soon as   the second point at the end of subsection \ref{sub} is satisfied) 
but the complexity 
 of the algorithm for the computation of the crystal isomorphism from ${\bf s}$ to a very dominant multicharge increases. However, It is not unreasonable to expect that some particular multicharge 
  can lead to interesting fast new algorithms. 
 
\subsection{Steps $1$ and $2$}\label{12}
It follows from Section \ref{crysi} that the first two steps can be both implemented by the process below. 
Let $0<s<e$ and set ${\bf s}=(0,s)$. We set ${\ulambda}[1]=(\lambda_1,\ldots,\lambda_{e-s})$ 
 and ${\ulambda}[2]=(\lambda_{e-s+1},\ldots,\lambda_{r})$, we write the Young tableau of $\ulambda[1]$ with the associated contents
  and just below, the Young tableau of $\ulambda[2]$ with the associated contents  with respect to the multicharge $(0,s)$. 
   $$
\begin{array}{|c|c|c|c|c|c|c|}
 \cline{1-1}
\ulambda[1]  \\
  \hline
  0&1  &2 & 3 &  \ldots  &\ldots  &  \lambda_1-1\\
   \hline
  \overline{1}&0 &1 & 2 & \ldots  & \lambda_2-2 \\
  \cline{1-6}   
    \vdots & \vdots & \vdots & \vdots &  \vdots  \\
        \cline{1-5} 
    \overline{e-s-1} & \ldots & \ldots & \lambda_{e-s}-(e-s)\\
        \cline{1-4} 
\\ 
\ulambda[2]  \\
        \cline{1-4} 
  s&s+1   & \ldots &   \lambda_{e-s+1}-1+s\\
  \cline{1-4}
  s-1&\ldots   & \lambda_{e-s+2}-2+s   \\
  \cline{1-3}   
\ldots & \ldots 
\end{array}$$
For example, take $\lambda=(10,8,7,5,4,4,3,2,1,1)$. Take $e=4$ and $s=1$ 
   $$
\begin{array}{|c|c|c|c|c|c|c|c|c|c|}
 \cline{1-1}
\ulambda[1]  \\
  \hline
  0&1  &2 & 3 &  4&5&{\bf 6}&{\bf 7}&{\bf 8}&{\bf 9}  \\
  \cline{1-10}
  \overline{1}&0 &1 & 2 &  3&{\bf 4}&{\bf 5}&{\bf 6} \\
  \cline{1-8}
 \overline{2} &   \overline{1}&0&1&2& {\bf 3}&{\bf 4} \\
  \cline{1-7}
\\ 
\ulambda[2]  \\
      \cline{1-5}
 1 & 2 &  3&4&5 \\     
       \cline{1-5}
 0&1&2& 3 \\
  \cline{1-4} 
  \overline{1}&0&1 & 2\\
\cline{1-4}
    \overline{2}&\overline{1}&0\\
  \cdashline{1-3}
      \overline{3} & \overline{1}   \\
  \cline{1-2}
      \overline{4}\\
    \cline{1-1}   
    \overline{5} \\
        \cline{1-1} 
\end{array}$$

Now, starting with the first part of $\ulambda[1]$,  consider the content of the rightmost box, say $c$. 
 In $\ulambda[2]$, we consider the rightmost boxes and we take the one  with the  greatest content   which is less than $c$, say $c'$. 
Then we remove the boxes of the first part of $\ulambda[1]$  
 with content greater than $c'$ into this part in $\ulambda[2]$ (in other words, we move the ``truncated first  row'' containing the boxes grater than $c$ to the row in $\ulambda[2]$).

  It is clear that we still have a partition.  
 Then,  we do the same for the second part of $\ulambda[1]$ and so on until we reach the last part of $\ulambda[1]$. If this is not possible we switch to the second part  of $\ulambda[1]$,  
   and we continue this process  until we reach the last part of $\ulambda[1]$. 
   
   In our example,  we must remove the boxes in bold in the first partition above, and add the boxes in bold in the second partition below. 
   $$
\begin{array}{|c|c|c|c|c|c|c|c|c|c|}
 \cline{1-1}
\ulambda[1]  \\
\cline{1-6}
  0&1  &2 & 3 &  4&5  \\
  \cline{1-6}
  \overline{1}&0 &1 & 2 &  3\\
  \cline{1-5}
 \overline{2} &   \overline{1}&0&1&2 \\
  \cline{1-5}
\\ 
\ulambda[2]  \\
      \cline{1-9}
 1 & 2 &  3&4&5& {\bf 6}&{\bf 7}&{\bf 8}&{\bf 9} \\     
       \cline{1-9}
 0&1&2& 3 &{\bf 4}&{\bf 5}&{\bf 6}  \\
  \cline{1-7} 
  \overline{1}&0&1 & 2& {\bf 3}&{\bf 4}\\
  \cline{1-6}   
    \overline{2}&\overline{1}&0\\
  \cdashline{1-3}
      \overline{3} & \overline{1}   \\
  \cline{1-2}
      \overline{4}\\
    \cline{1-1}   
    \overline{5} \\
        \cline{1-1} 
\end{array}$$

We then collect all the parts of $\ulambda[2]$ that are above the smallest part we have modified, 
 in a partition $\mu$. So here $\mu=(9,7,6)$.   The new partition $\ulambda[2]$ is given by the remaining parts and we add $e$ to the contents of all the boxes in it. 
  We then move the step above and continue the process until we cannot do anything.   The remaining parts of $\ulambda[2]$ are added to $\mu$. Then 
   the partition $\ulambda[1]$ is the first component of  $\Psi^{(0,s)\to (0,s+k.e)}_e (\lambda^1,\lambda^2)$ and $\mu$ is the second. 

   $$
\begin{array}{|c|c|c|c|c|c|c|}
 \cline{1-1}
\ulambda[1]  \\
\cline{1-6}
  0&1  &2 & 3 &  4&{\bf 5}  \\
  \cline{1-6}
  \overline{1}&0 &1 & 2 &  {\bf 3}\\
  \cline{1-5}
 \overline{2} &   \overline{1}&0&{\bf 1}&{\bf 2} \\
  \cline{1-5}
\\ 
\ulambda[2]  \\
  \cline{1-3}   
  2&3&4\\
      \cline{1-3}   
     1 & 2   \\
  \cline{1-2}
    0\\
    \cline{1-1}   
    \overline{1} \\
        \cline{1-1} 
\end{array}$$
It becomes :

   $$
\begin{array}{|c|c|c|c|c|c|}
 \cline{1-1}
\ulambda[1]  \\
\cline{1-5}
  0&1  &2 & 3 &  4 \\
  \cline{1-5}
  \overline{1}&0 &1 & 2 \\
  \cline{1-4}
 \overline{2} &   \overline{1}&0\\
  \cline{1-3}
\\ 
\ulambda[2]  \\
  \cline{1-4}   
  2&3&4& {\bf 5}\\
      \cline{1-4}   
     1 & 2  &  {\bf 3} \\
  \cline{1-3}
    0 & {\bf 1}&{\bf 2} \\
    \cline{1-3}   
    \overline{1} \\
        \cline{1-1} 
\end{array}$$
We have  now $\mu=(9,7,6,4,3,3)$, and we the above process:

   $$
\begin{array}{|c|c|c|c|c|c|}
 \cline{1-1}
\ulambda[1]  \\
\cline{1-5}
  0&1  &2 & 3 &  {\bf 4} \\
  \cline{1-5}
  \overline{1}&0 &1 & {\bf 2} \\
  \cline{1-4}
 \overline{2} &   \overline{1}&0\\
  \cline{1-3}
\\ 
\ulambda[2]  \\
    \cline{1-1}   
3 \\
        \cline{1-1} 
\end{array}$$
gives :
   $$
\begin{array}{|c|c|c|c|c|c|}
 \cline{1-1}
\ulambda[1]  \\
\cline{1-4}
  0&1  &2 & 3  \\
  \cline{1-4}
  \overline{1}&0 &1  \\
  \cline{1-3}
 \overline{2} &   \overline{1}&0\\
  \cline{1-3}
\\ 
\ulambda[2]  \\
    \cline{1-2}   
3 & {\bf 4} \\
        \cline{1-2} 
        {\bf 2} \\
       \cline{1-1}       
\end{array}$$
and then $\mu=(9,7,6,4,3,3,2,1)$. There is nothing we can do now. the bipartition we are searching for is $((4,3,3),((9,7,6,4,3,3,2,1))$

\subsection{Step $3$ and $4$}\label{34}

At this stage, we have computed $(\mu^1,\mu^2):=\Psi^{(0,s)\to (0,s+k.e)}_e (\lambda^1,\lambda^2)$ . By induction, we thus know
 $(\nu^1,\nu^2):=(m_e (\mu^1),m_e(\mu^2))$ and we must do the reversed process as the one above
  to get our bipartition:
  $$\Psi_e^{(0,-s+ke) \to  ( (0,e-s)} (\nu^1,\nu^2).$$  
  This is 
 done as follows.  
 
 We write the Young tableau of  $\nu^1$ with the associated contents for each box,
  and just below, the Young tableau of $\nu^2$ with the associated contents charged by $ke-s$ where $k$ is sufficiently large (that is, the content of the box $(a,b)$ is $b-a+(ke-s)$). Keeping the above example, we have by induction 
  $m_4 (4,3,3)=(10)$ and $m_4 (9,7,6,4,3,3,2,1)=(14,7,7,3,3,1)$. So we consider the bipartition  $((10),(14,7,7,3,3,1)$ and the multicharge is $(0,3)$. 
   $$
\begin{array}{|c|c|c|c|c|c|c|c|c|c|c|c|c|c|}
 \cline{1-1}
\nu^1  \\
\cline{1-10}
0 & 1 &2&3&4&5&6&7&8&9\\
\cline{1-10}
\\ 
\nu^2  \\
      \cline{1-14}
 19 & 20 &  21&22&23& 24&25&26&27 & 28 & 29 & 30 & 31& 32\\     
       \cline{1-14}
 18&19&20& 21 &22&23&24  \\
  \cline{1-7} 
 17&18&19 & 20& 21&22& 23\\
  \cline{1-7}   
16&17&18\\
      \cline{1-3}   
15& 16 & 17  \\
  \cline{1-3}
14  \\
        \cline{1-1} 
\end{array}$$

At each step, starting from the bottom of $\nu^2$, we see if one  can remove boxes from $\nu^2$ to add it to $\nu^1$ as in the subsection above (except that we remove the box from the other partition).  Note that  $\nu^1$ need to always have  the same number of rows so we only add the possible boxes in the $e-s$ rows of $\nu^1$.  
Then we remove $e$ from all the contents of the boxes of $\nu^2$.  
 In the example, 
we have nothing to do so we remove $e$ from all the contents of the second partitions and again one more time. 

   $$
\begin{array}{|c|c|c|c|c|c|c|c|c|c|c|c|c|c|}
 \cline{1-1}
\nu^1  \\
\cline{1-10}
0 & 1 &2&3&4&5&6&7&8&9\\
\cline{1-10}
\\ 
\nu^2\\
      \cline{1-14}
15&16& 17 &18&19&20 &21&22&23 & 24 & 25 & 26 & 27& 28\\     
       \cline{1-14}
14&15&16& 17 &18&19&20  \\
  \cline{1-7} 
 13&14&15 & 16& 17&18& 19\\
  \cline{1-7}   
12&13&14\\
      \cline{1-3}   
11& 12 & 13  \\
  \cline{1-3}
10  \\
        \cline{1-1} 
\end{array}$$

Then we can add a box of content $10$  and we subsract $e$ from all the contents.  We  then successively obtain the following bipartitions. 
   $$
\begin{array}{|c|c|c|c|c|c|c|c|c|c|c|c|c|c|}
 \cline{1-1}
\nu^1 \\
\cline{1-11}
0 & 1 &2&3&4&5&6&7&8&9& {\bf 10}\\
\cline{1-11}
\\ 
\nu^2 \\
      \cline{1-14}
15&16& 17 &18&19&20 &21&22&23 & 24 & 25 & 26 & 27& 28\\     
       \cline{1-14}
14&15&16& 17 &18&19&20  \\
  \cline{1-7} 
 13&14&15 & 16& 17&18& 19\\
  \cline{1-7}   
12&13&14\\
      \cline{1-3}   
11& 12 & 13  \\
  \cline{1-3}
\end{array}$$
and then:
   $$
\begin{array}{|c|c|c|c|c|c|c|c|c|c|c|c|c|c|c|c|}
 \cline{1-1}
\nu^1\\
\cline{1-12}
0 & 1 &2&3&4&5&6&7&8&9& {\bf 10} & {\bf 11}\\
\cline{1-12}
\\ 
\nu^2  \\
      \cline{1-14}
3& 4& 5&6& 7&8 &9&{10}& 11 & 12 & 13 & 14 & 15& 16\\     
       \cline{1-14}
2&3 & 4& 5&6& 7&8  \\
  \cline{1-7} 
 1&2&3 & 4& 5&6\\
  \cline{1-6}   
0&1&2\\
      \cline{1-3}   
\overline{1}& 0 & 1  \\
  \cline{1-3}
 
\end{array}$$
and then:
   $$
\begin{array}{|c|c|c|c|c|c|c|c|c|c|c|c|c|c|c|c|c|c|}
 \cline{1-1}
\nu^1 \\
\cline{1-17}
0 & 1 &2&3&4&5&6&7&8&9& {10}& 11 & {\bf 12} & {\bf 13} & {\bf 14} & {\bf 15}& {\bf 16}\\
      \cline{1-17}   \\ 
\nu^2 \\
      \cline{1-9}
3 & 4& 5&6& 7&8 &9&10&11\\     
       \cline{1-9}
2&3 & 4& 5&6& 7&8  \\
  \cline{1-7} 
 1&2&3 & 4& 5&6\\
  \cline{1-6}   
0&1&2\\
  \cline{1-3}   
{ \overline{1}}& { 0} & { 1}  
\\ 
      \cline{1-3} 

  \cline{1-3}
 
\end{array}$$
At the end, the  concatenation (and reordering the parts if necessary)  of the two partitions we get must be  $m_e (\lambda)$. 
In our example, we obtain $((17),(9,7,6,3,3))$ so that $m_e (\lambda)=(17,9,7,6,3,3)$. 

\subsection{Example}
Let us keep our running example $\lambda=(10,8,7,5,4,4,3,2,1,1)$, $l=2$ and $e=4$ but this time, we take  $s=2$.  
The first two steps will give:
{\small
   $$
\begin{array}{|c|c|c|c|c|c|c|c|c|c|}
 \cline{1-1}
\ulambda[1]  \\
  \hline
  0&1  &2 & 3 &  4&5& 6&{ 7}&{  8}&{\bf 9}  \\
  \cline{1-10}
  \overline{1}&0 &1 & 2 &  3&{  4}&{  5}&{\bf 6} \\
  \cline{1-8}
\\ 
\ulambda[2]  \\
  \cline{1-7}
2&   3&4&5&6& 7&8 \\
      \cline{1-7}
 1 & 2 &  3&4&5 \\     
       \cline{1-5}
 0&1&2& 3 \\
  \cline{1-4} 
  \overline{1}&0&1 & 2\\
\cline{1-4}
    \overline{2}&\overline{1}&0\\
  \cline{1-3}
      \overline{3} & \overline{1}   \\
  \cline{1-2}
      \overline{4}\\
    \cline{1-1}   
    \overline{5} \\
        \cline{1-1} 
\end{array}
\to
\begin{array}{|c|c|c|c|c|c|c|c|c|}
 \cline{1-1}
\ulambda[1]  \\
  \hline
  0&1  &2 & 3 &  4&5& 6&{ 7}&{  \bf 8}\\
  \cline{1-9}
  \overline{1}&0 &1 & 2 &  3&{  4}&{  \bf 5} \\
  \cline{1-7}
\\ 
\ulambda[2]  \\
       \cline{1-4}
 4&5&6& 7 \\
  \cline{1-4} 
3&4&5 & 6\\
\cline{1-4}
2&3&4\\
  \cline{1-3}
   1 & 2   \\
  \cline{1-2}
0\\
    \cline{1-1}   
    \overline{1} \\
        \cline{1-1} 
\end{array}
\to
\begin{array}{|c|c|c|c|c|c|c|c|}
 \cline{1-1}
\ulambda[1]  \\
  \hline
  0&1  &2 & 3 &  4&5& 6&{ \bf 7}\\
  \cline{1-8}
  \overline{1}&0 &1 & 2 &  3&{  4}\\
  \cline{1-6}
\\ 
\ulambda[2]  \\
  \cline{1-2}
   5 & 6   \\
  \cline{1-2}
4\\
    \cline{1-1}   
 3 \\
        \cline{1-1} 
\end{array}
$$}
$$\to
\begin{array}{|c|c|c|c|c|c|c|c|}
 \cline{1-1}
\ulambda[1]  \\
  \hline
  0&1  &2 & 3 &  4&5& {\bf 6}\\
  \cline{1-7}
  \overline{1}&0 &1 & 2 &  3\\
  \cline{1-5}
\\ 
\ulambda[2]  \\
  \cline{1-1}
8\\
    \cline{1-1}   
 7 \\
        \cline{1-1} 
\end{array}$$
and thus, we obtain the bipartition $((6,6),(8,6,5,4,4,3,1,1,1)$ which is thus the bipartition 
$$\Psi^{(0,2)\to (0,2+4k)}_4 ((10,8,3,2,1,1),(7,5,4,4)).$$ Now, by induction, we know 
 $m_4 (6,6)=(6,6)$ and $m_4 (8,6,5,4,4,3,1,1,1)=(15,7,5,4,1,1)$.  We now perform Steps $3$ and $4$ for $((6,6),(15,7,5,3,1,1))$. 
   {\tiny$$
\begin{array}{|c|c|c|c|c|c|c|c|c|c|c|c|c|c|c|}
 \cline{1-1}
\nu^1  \\
\cline{1-6}
0 & 1 &2&3&4&5\\
\cline{1-6}
\overline{1} & 0 & 1 &2&3&4\\
\cline{1-6}
\\ 
\nu^2  \\
      \cline{1-15}
 10 & 11 &  12&13&14& 15&16&17&18 & 19 & 20 & 21 & 22& 23& 24\\     
       \cline{1-15}
 9&10&11& 12 &13&14&15  \\
  \cline{1-7} 
 8&9&10 & 11& 12\\
  \cline{1-5}   
7&8&9 & 10\\
      \cline{1-4}   
{\bf 6} \\
  \cline{1-1}
{\bf 5}  \\
        \cline{1-1} 
\end{array}
\to 
\begin{array}{|c|c|c|c|c|c|c|c|c|c|c|c|c|c|c|}
 \cline{1-1}
\nu^1  \\
\cline{1-7}
0 & 1 &2&3&4&5 & 6\\
\cline{1-7}
\overline{1} & 0 & 1 &2&3&4 & 5\\
\cline{1-7}
\\ 
\nu^2  \\
      \cline{1-15}
6& 7 &  8&9&10& 11&12&13&14 & 15 & 16 & 17 & 18& 19& 20\\     
       \cline{1-15}
 5&6&7& 8 &9&10&11  \\
  \cline{1-7} 
 4&5&6&{\bf  7}& {\bf 8}\\
  \cline{1-5}   
3&4&5 & {\bf 6}\\
      \cline{1-4}   
\end{array}
$$}
 {\tiny
$$
\begin{array}{|c|c|c|c|c|c|c|c|c|c|c|c|c|c|c|c|c|}
 \cline{1-1}
\nu^1  \\
\cline{1-9}
0 & 1 &2&3&4&5 & 6 & 7 & 8\\
\cline{1-9}
\overline{1} & 0 & 1 &2&3&4 & 5& 6\\
\cline{1-8}
\\ 
\nu^2  \\
      \cline{1-15}
2& 3 &  4&5&6& 7&8& {\bf 9} & {\bf 10} & {\bf 11}& {\bf 12} & {\bf 13} &  {\bf14 } &  {\bf15 } \\     
       \cline{1-15}
 1&2&3& 4 &5&6  \\
  \cline{1-6}  
\end{array}
\to  
\begin{array}{|c|c|c|c|c|c|c|c|c|c|c|c|c|c|c|c|c|}
 \cline{1-1}
\nu^1  \\
\hline
0 & 1 &2&3&4&5 & 6 & 7 & 8& 9&10 & 11 & 12 & 13 & 14& 15\\
\cline{1-17}
\overline{1} & 0 & 1 &2&3&4 & 5& 6 & 7 \\
\cline{1-9}
\\ 
\nu^2  \\
      \cline{1-1}
\end{array}
$$
}
We obtain the bipartition $((17,9),(7,6,3,3))$ and we conclude that  $m_e (\lambda)=(17,9,7,6,3,3)$ as in the last section. 
\section{Xu's algorithm}\label{alg}
In \cite{Xu1,Xu2}, Xu has given an algorithm for the computation of the Mullineux involution which is derived from the original Mullineux's algorithm. We here recall 
 this algorithm and then show that it can be seen as a particular case of ours. This will in particular give a new elementary proof for the fact that the algorithm computes the Mullineux involution.
 
 \subsection{The algorithm}
To describe Xu's algorithm, we will need some additional combinatorial definitions. Let $\lambda=(\lambda_1,\ldots,\lambda_r)$ be an $e$-regular partition with $\lambda_r\neq 0$.  The {\it  rim} of $\lambda$ is 
the subset of the Young diagram of $\ulambda$ consisting in the $(i,j)$ such that $(i+1,j+1)$ is not in $[\lambda]$.  The {\it $e$-rim} is now 
 the subset $\{ (a_1,b_1),\ldots,(a_m,b_m)\}$ of the rim of $\lambda$ which is obtained by following the rim of $\lambda$ from right to left and top to bottom,  and moving down one row every time the number of nodes we have  is dividible by $e$. 
 \begin{exa}
 Let $e=3$ and $\lambda=(7.4.2.2)$. The $e$-rim is given by the nodes marked by a star. 
 
   $$
\begin{array}{|c|c|c|c|c|c|c|c|}
  \hline
&  & &   & &\bigstar& \bigstar& \bigstar  \\
     \hline
&   & \bigstar&  \bigstar & \bigstar \\
   \cline{1-5}
&     & \bigstar\\
       \cline{1-3}
\ \  &   \bigstar    & \bigstar \\
          \cline{1-3}
\end{array}
 $$
 
 \end{exa}
 Assume that the cardinality of the $e$-rim of $\lambda$ is $m$. 
 The {\it truncated $e$-rim} of $\lambda$ is by definition the set of nodes $(i,j)$ in the $e$-rim of $\lambda$ such that $(i,j-1)$ is also in the $e$-rim of $\lambda$. 
  If $e$ does not divide $m$, we add also the node 
 $(r,x)$ in the $e$-rim of $\lambda$ such that $(r,x-1)$ is not in the $e$-rim.  We now define $\widetilde{\lambda}$  to be the partition  obtained by removing the truncated $e$-rim from $\lambda$. 
 It is easy to see that this partition is $e$-regular with rank strictly less than the rank of $\lambda$. 
  \begin{exa}
 Let $e=3$ and $\lambda=(8,5,3,3)$. The truncated $e$-rim is given by the nodes marked by a star. 
 
   $$
\begin{array}{|c|c|c|c|c|c|c|c|}
  \hline
&  &\ \  &   & &\ \ & \bigstar& \bigstar  \\
     \hline
&   & &  \bigstar & \bigstar \\
   \cline{1-5}
&     &\\
       \cline{1-3}
\ \  &   \ \   & \bigstar \\
          \cline{1-3}
\end{array}
 $$
 So the partition $\widetilde{\lambda}$  is $(6,3,3,2)$. 
 \end{exa}
Now we define a map 
$$X_e : \Phi_{(e,(0))} \to 
 \Phi_{(e,(0))} $$
 recursively as follows. We define $X_e (\emptyset)=\emptyset$ and if $\lambda\in  \Phi_{(e,(0)}(n)$ with $n\neq 0$ then $X_e (\lambda)$ is obtained by adding a column of length 
  $n-|\widetilde{\lambda} |$ to $X_e (\widetilde{\lambda} )$. 
  \begin{Th}[Xu]
  We have $X_e=m_e$. 
  \end{Th}
  
  We here give a new proof of this Theorem using the crystal isomorphisms. 
  
\begin{exa}
We keep the above example. We can compute $X_3 (6,3,3,2)=(8,2,2,1,1)$, now we have exactly  $5$ nodes in the truncated $3$-rim of $\lambda$ so
 $X_e (7,4,2,2)$ is obtained by adding a column of length $5$ to $(8,2,2,1,1)$ and we get $X_3 (8,5,3,3)=(9,3,3,2,2)$. 

\end{exa}

\subsection{Relation with crystal isomorphisms}
We  will see in this subsection that Xu's algorithm is equivalent to ours in  the case where we choose $s=e-1$. For $\lambda$ an  $e$-regular  partition, we denote by $\widetilde{\lambda}$ 
 the partition obtained by removing the truncated $p$-rim as in Xu's algorithm.  We denote by $r$ the number of boxes in the truncated $p$-rim. 
  \begin{Prop}\label{12p}
  
  We have $\Psi_e^{(0,e-1)\to (0,e-1+ke)} \circ \theta_{e,(0,e-1)} (\lambda)=(r,\widetilde{\lambda})$ ($k>>0)$
  \end{Prop}
\begin{proof}
We denote $\ulambda[2]=(\lambda_2,\ldots,\lambda_e)$. 
We begin with the two first steps of our algorithm which  are described in \S \ref{12}. Assume first that one cannot add any ``truncated row'' of $\lambda_1$  in $\ulambda[2]$. 
 This means that  there exists $k>0$  such that $\lambda_{k+1}-k+e-1=\lambda_1-1$ and we have the following partitions:

   $$
\begin{array}{|c|c|c|c|c|c|c|c|c|c|c|}
 \cline{1-1}
\ulambda[1]  \\
  \hline
  0&1  &2 &  \ldots  &\ldots  &\ldots  &\ldots  &\ldots  & \textcolor{red}{x} & \textcolor{red}\ldots  & \textcolor{red}{ \lambda_1-1}\\
   \hline
\\ 
\ulambda[2]  \\
        \cline{1-9} 
  e-1&e   & \ldots & \ldots &   \ldots &   \ldots &   \ldots &    \textcolor{blue}{\ldots} &   \textcolor{blue}{ x+e-1}\\
         \cline{1-9} 
  e-2&e-1   & \ldots &   \ldots &  \ldots &   \ldots &     \textcolor{blue}{\ldots} &     \textcolor{blue}{\lambda_{3}+e-3}\\     
             \cline{1-8}   
               \vdots & \vdots & \vdots & \vdots &  \vdots &   \vdots   &   \textcolor{blue}{ \vdots} \\
           \cline{1-7}   
             e-s+1& \ldots & \ldots & \ldots &  \ldots&   \textcolor{blue}{\ldots}   &  \textcolor{blue}{\lambda_{s}-s+e}\\
                        \cline{1-7}   
                        \ldots  & \ldots & \ldots & {\ldots} & \vdots  \\
                           \cline{1-5}   
e-k  & \ldots & \ldots &{\ldots} & \textcolor{red}{\lambda_1-1} \\
                      \cline{1-5}   
e-k-1  & \ldots & \ldots & {\ldots} & \textcolor{red}{\lambda_1-2} \\
         \cline{1-5}   
\end{array}$$
(with $x=\lambda_2-1$)

Then the partition $(\lambda_2,\ldots,\lambda_{k+1})$ corresponds to  the partition $(\lambda_1,\ldots,\lambda_k)$ with  the very first truncated $e$-rim removed.
 If $\lambda_{k+1}=0$ then we are done and $\lambda_1$ is the number of nodes in the truncated $p$-rim minus $1$. In this case the number of elements in the associated $e$-rim is not $e$. 
Otherwise we get $e$ boxes in the associated rim and we must go to the second step of our algorithm.

Assume that one can add a truncated row of length $r$. Assume that the  row is added in    the part $\lambda_{k+1}$. Then the partition $(\lambda_2,\ldots,\lambda_{k+1})$ corresponds to  the partition $(\lambda_1,\ldots,\lambda_k)$ with  a truncated $e$-rim removed. 

   $$
\begin{array}{|c|c|c|c|c|c|c|c|c|c|c|}
 \cline{1-1}
\ulambda[1]  \\
  \hline
  0&1  &2 &  \ldots  &\ldots  &\ldots  &\ldots  &\ldots  & \textcolor{red}{x} & \textcolor{red}\ldots  & \textcolor{red}{ \lambda_1-1}\\
   \hline
\\ 
\ulambda[2]  \\
        \cline{1-9} 
  e-1&e   & \ldots & \ldots &   \ldots &   \ldots &   \ldots &    \textcolor{blue}{\ldots} &   \textcolor{blue}{ x+e-1}\\
         \cline{1-9} 
  e-2&e-1   & \ldots &   \ldots &  \ldots &   \ldots &     \textcolor{blue}{\ldots} &     \textcolor{blue}{\lambda_{3}+e-3}\\     
             \cline{1-8}   
               \vdots & \vdots & \vdots & \vdots &  \vdots &   \vdots   &   \textcolor{blue}{ \vdots} \\
           \cline{1-7}   
             e-k+1& \ldots & \ldots & \ldots &   \textcolor{blue}{\lambda_1} &   \textcolor{blue}{\ldots}   &  \textcolor{blue}{\lambda_{k}-k+e}\\
                        \cline{1-7}   
e-k  & \ldots & \textcolor{red}{x} & \textcolor{red}{\ldots} & \textcolor{red}{\lambda_1-1} \\
         \cline{1-5}   
\end{array}$$
If $\lambda_{k+1}$ is non zero. Note that the length of the truncated $p$-rim is $e-k$.  
 By induction, the first $e-k$ nodes of the partition $\ulambda[1]$ will not moved in our algorithm
We can thus just argue by induction by replacing $\ulambda[1]$ with the partition
$\ulambda[1]- (e-k)$ to find $\ulambda[2]$  and take into account that we must add $e-k$ (the length of the truncated $p$-rim) to the 
 partition we obtain at the end of our algorithm. Note that the content of the leftmost node in  our first partition will be now $e-k$ and 
  the contet of the leftmost node of the second partition $(e-k)-e-1$ so the induction can be done.

\end{proof}
On the other hand, we now have the following result:
  \begin{Prop}\label{13}
  
  We have $\Psi_e^{(0,1)\to (0,1+ke)} \circ \theta_{e,(0,1)} (\lambda)=(m_e (t) , \lambda-1)$ where $t$ is the length of the first column of $\lambda$ ($k>>0$)
  \end{Prop}
\begin{proof}
We use the algorithm described in subsection \ref{34}, using these notations, we are in the following configuration:
   $$
\begin{array}{|c|c|c|c|c|c|c|}
 \cline{1-1}
\ulambda[1]  \\
  \hline
  0&1  &2 & 3 &  \ldots  &\ldots  &  \lambda_1-1\\
   \hline
  \overline{1}&0 &1 & 2 & \ldots  & \lambda_2-2 \\
  \cline{1-6}   
    \vdots & \vdots & \vdots & \vdots &  \vdots  \\
        \cline{1-5} 
    \overline{e-2} & \ldots & \ldots & \lambda_{e-1}-(e-1)\\
        \cline{1-4} 
\\ 
\ulambda[2]  \\
        \cline{1-5} 
  1&2   & \ldots & \ldots &   \lambda_{e}\\
  \cline{1-5}
  0&\ldots   & \ldots & \lambda_{e+1}-1   \\
  \cline{1-4}   
  \vdots & \vdots &  \vdots &  \vdots  \\
    \cline{1-4}  
    \overline{e-3}&\ldots   & \ldots   &  \lambda_{2e-2}-(e-2)   \\
        \cline{1-4}  
            \overline{e-2}&\ldots   &  \lambda_{2e-1}-(e-1)   \\
                   \cline{1-3}  
\vdots & \vdots &\vdots  
\end{array}$$
The first step of our algorithm thus gives:

   $$
\begin{array}{|c|c|c|c|c|c|c|}
 \cline{1-1}
\ulambda[1]  \\
  \hline
  0&1  &2 & 3 &  \ldots  &\ldots  &  \lambda_e\\
   \hline
  \overline{1}&0 &1 & 2 & \ldots  & \lambda_{e+1}-1 \\
  \cline{1-6}   
    \vdots & \vdots & \vdots & \vdots &  \vdots  \\
        \cline{1-5} 
    \overline{e-2} & \ldots & \ldots & \lambda_{2e-2}-(e-2)\\
        \cline{1-4} 
\\ 
\ulambda[2]  \\
        \cline{1-5} 
  1&2   & \ldots & \ldots &   \lambda_{1}-1\\
  \cline{1-5}
  0&\ldots   & \ldots & \lambda_{2}-2   \\
  \cline{1-4}   
  \vdots & \vdots &  \vdots  \\
    \cline{1-3}  
    \overline{e-3}&\ldots   &  \lambda_{e-1}-(e-1)   \\
        \cline{1-3}  
            \overline{e-2}&\ldots   &  \lambda_{2e-1}-(e-1)   \\
        \cline{1-3}  
\vdots & \vdots &\vdots  
\end{array}$$
and  now, we have to perform the algorithm for the following configuration of partitions:
   $$
\begin{array}{|c|c|c|c|c|c|c|}
 \cline{1-1}
\ulambda[1]  \\
  \hline
  0&1  &2 & \ldots &  \ldots  &\ldots  &  \lambda_e\\
   \hline
  \overline{1}&0 &1 & \ldots & \ldots  & \lambda_{e+1}-1 \\
  \cline{1-6}   
    \vdots & \vdots & \vdots & \vdots &  \vdots  \\
        \cline{1-5} 
    \overline{e-2} & \ldots & \ldots & \lambda_{2e-2}-(e-2)\\
        \cline{1-4} 
\\ 
\ulambda[2]'  \\
        \cline{1-5} 
  2& & \ldots & \ldots &   \lambda_{2e-1}+1\\
  \cline{1-5}
  1&\ldots   & \ldots & \lambda_{2e}   \\
  \cline{1-4}   
  \vdots & \vdots &  \vdots  \\
    \cline{1-3}  
    \overline{e-2}&\ldots   &  \lambda_{3e-3}-e+3   \\
        \cline{1-3}  
\vdots & \vdots &\vdots  
\end{array}$$
which thus leads to 
   $$
\begin{array}{|c|c|c|c|c|c|c|}
 \cline{1-1}
\ulambda[1]  \\
  \hline
  0&1  &2 & 3 &  \ldots  &\ldots  &  \lambda_{2e-1}+1\\
   \hline
  \overline{1}&0 &1 & 2 & \ldots  & \lambda_{2e} \\
  \cline{1-6}   
    \vdots & \vdots & \vdots & \vdots &  \vdots  \\
        \cline{1-5} 
    \overline{e-2} & \ldots & \ldots & \lambda_{3e-3}-e+3\\
        \cline{1-4} 
\\ 
\ulambda[2]'  \\
        \cline{1-5} 
  2& & \ldots & \ldots &   \lambda_{e}\\
  \cline{1-5}
  1&\ldots   & \ldots & \lambda_{e+1}-1   \\
  \cline{1-4}   
  \vdots & \vdots &  \vdots  \\
    \cline{1-3}  
    \overline{e-2}&\ldots   &  \lambda_{2e-2}-(e-2)   \\
        \cline{1-3}  
     \overline{e-3}&\ldots   &  \lambda_{3e-2}-(e-2)   \\       
         \cline{1-3}       
\vdots & \vdots &\vdots  
\end{array}$$

Now, we come to the last step, assume that  $s$ is maximal such that $\lambda_{ke-k+s}\neq 0$ (so that $ke-k+s$ is the length of the first column of $\lambda$). 
    Then, we are in the following configuration where 
 we have an addable $k+1$-node in the second partition.


   $$
\begin{array}{|c|c|c|c|c|c|}
 \cline{1-1}
\ulambda[1]  \\
  \hline
  0&1   & \ldots &  \ldots  &\ldots  &  \lambda_{ke-k+1}+k-1\\
   \hline
  \overline{1}&0  & \ldots  & \ldots  & \lambda_{ke-k+2}+k-2 \\
  \cline{1-5}   
    \vdots & \vdots& \vdots &  \vdots  \\
        \cline{1-4} 
    \overline{s-1} & \ldots & \ldots & \lambda_{ke-k+s}+k-s\\
            \cline{1-4} 
                \overline{s} & \ldots & k-s-1\\
        \cline{1-3} 
    \vdots & \vdots& \vdots  \\
        \cline{1-3} 
    \overline{e-2} & \ldots & k-(e-1)\\
        \cline{1-3} 
\\ 
\ulambda[2]'  \\ 
        \cline{1-1}  
\end{array}$$
and we obtain  for $\ulambda[1]$:

   $$
\begin{array}{|c|c|c|c|c|}
  \hline
  0 & \ldots &  \ldots  &k  \\
   \hline
  \overline{1}&0  & \ldots  & k-1  \\
  \cline{1-5}   
    \vdots & \vdots& \vdots &  \vdots  \\
        \cline{1-4} 
    \overline{s-1} & \ldots & k-s &k+1-s\\
            \cline{1-4} 
                \overline{s} & \ldots & k-s-1\\
        \cline{1-3} 
    \vdots & \vdots& \vdots  \\
        \cline{1-3} 
    \overline{e-2} & \ldots & k-(e-1)\\
        \cline{1-3} 
\end{array}$$

The first partition is the Mullineux image of the partition $((k+1)(e-1)-(-s-2+e+1))=(ke-k+s)$. The second partition we get in the algorithm  is $\lambda-1$ which is exactly what we wanted.

\end{proof}

Let us now explain in which way our two algorithms are equivalent in the case where we choose ${\bf s}=(0,e-1)$. Let $\lambda$ be an $e$-regular partition and recall the $4$ steps of our algorithm at the beginning of \S \ref{stepa}. 
\begin{enumerate}
\item By Proposition \ref{12p}, after the two first steps of our algorithm, we obtain $(r,\widetilde{\lambda})$ where $r$ is the number of boxes in the truncated rim. 
\item By induction, we know $m_e (\widetilde{\lambda})$  and  the third step of our algorithm consists in the computation of  
 the image of 
$(m_e (r),m_e (\widetilde{\lambda}))$ with respect to $\Psi_e^{ (0,1+ke)\to (0,1)}$ (for $k>>0$). 
\item By Proposition \ref{13} that we apply to $\mu=m_e (\lambda)$, we have
$\Psi_e^{(0,1)\to (0,1+ke)} \circ \theta_{e,(0,1)} (\mu)=(m_e (t) , \mu-1)$ (where $t$ is the length of the first column of $\mu$) 
 so $m_e (\lambda)$ is the partition obtained by adding a row of length $r$ to $m_e (\widetilde{\lambda})$ as in Xu's algorithm. 

\end{enumerate}

The above result thus shows that Xu's algorithm indeed computes the Mullineux involution. 

\begin{Rem}
In \cite{BK}, Brundan and Kujawa gave another interpretation of the Xu's algorithm using the representation theory of the supergroup 
$GL(n|n)$. It would be interesting to understand the connection wof this work with ours. 

\end{Rem}

\noindent {\bf Address}\\

\noindent \textsc{Nicolas Jacon}, Universit\'e de Reims Champagne-Ardenne, UFR Sciences exactes et naturelles, Laboratoire de Math\'ematiques UMR CNRS 9008
Moulin de la Housse BP 1039, 51100 Reims, FRANCE\\  \emph{nicolas.jacon@univ-reims.fr}

\end{document}